\documentclass[10pt,reqno]{amsart}
\usepackage{bbm}
\usepackage{mathrsfs}
\usepackage{amsfonts} 
\usepackage[dvipsnames,usenames]{color}
\usepackage{graphicx,latexsym,bm,amsmath,amssymb,verbatim,multicol,lscape}
\usepackage{pgfplots}

\newtheorem{thm}{Theorem} [section]

\newtheorem{lem}[thm]{Lemma}
\newtheorem{exm}[thm]{Example}

\theoremstyle{definition}
\newtheorem{defn}[thm]{Definition} 
\theoremstyle{remark}

\numberwithin{equation}{section}

\begin{document}
\title[On Igusa local zeta functions of Hauser hybrid polynomials]
{On Igusa local zeta functions of Hauser hybrid polynomials}%
\author[S.F. Hong]{Shaofang Hong}
\address{Mathematical College, Sichuan University, Chengdu 610064, P.R. China}
\email{sfhong@scu.edu.cn, s-f.hong@tom.com, hongsf02@yahoo.com }
\author[Q.Y. Yin]{Qiuyu Yin}
\address{Mathematical College, Sichuan University, Chengdu 610064, P.R. China}
\address{Current address: Center for Combinatorics, LPMC-TJKLC, Nankai University,
Tianjin 300071, P.R. China}
\email{yinqiuyu26@126.com}
\thanks{S.F. Hong is the corresponding author and was supported partially
by National Science Foundation of China Grant \#11771304.}
 
\date{\today}%
\begin{abstract}
Let $K$ be a local field and $f(x)\in K[x]$ be a non-constant
polynomial. When ${\rm char}K=0$, Igusa showed the local zeta
function is a rational function. However, when ${\rm char}K>0$,
the rationality of the local zeta function is unknown in general.
In this paper, we study the local zeta functions for the so-called
hybrid polynomials in three variables with coefficients in a
non-archimedean local field of positive characteristic. These
hybrid polynomials were first introduced by Hauser in 2003 to
study the resolution of singularities in positive characteristic.
We establish the rationality theorem for these local zeta functions 
and list explicitly all the candidate poles. Our result generalizes 
the work of Le$\acute{o}$n-Cardenal, Ibadula and Segers and that
of Yin and Hong.
\end{abstract}

\maketitle

\section{\bf Introduction and the main results}
Let $K$ be a local field and $f(x)\in K[x]$ be a non-constant
polynomial. The local zeta function of $f$ was first introduced
by Weil \cite{[W]}. Later on, in the 1970's, Igusa \cite{[Igu1], [Igu2]}
systematically investigated this local zeta function. Lots of basic
results about local zeta functions was archived by Igusa. Meanwhile,
Igusa also made several conjectures that inspired many further
research in the past decades. In this paper, our main goal
is to study the rationality of Igusa local zeta function. For
this purpose, we first present the definition of the local zeta function.

Now let $K$ be a non-archimedean local field with $\mathcal{O}_K$ as
its ring of integers, that is, $\mathcal{O}_K$ is the maximal compact
subring of $K$. Let $\mathcal{O}_K^{\times}$ be the group of units of
$\mathcal{O}_K$ and let $\mathcal{P}_K$ be the unique maximal ideal
of $\mathcal{O}_K$. We fix an element $\pi\in K$ such that
$\mathcal{P}_K=\pi\mathcal{O}_K$. Let
$\mathbb{F}_q\cong\mathcal{O}_K/\mathcal{P}_K$
be the residue field of $K$ which is the finite field with
$q={\rm Card}(\mathcal{O}_K/\mathcal{P}_K)$ elements.
For $x\in K$, we define the {\it valuation} of $x$, denoted by
${\rm ord}(x)$, to be such that ${\rm ord}(\pi)=1$. We define
${\rm ord}(0):=\infty$. Then $x$ can be written uniquely as
$ac(x)\pi^{{\rm ord}(x)}$, where $ac(x)\in\mathcal{O}_K^{\times}$
is called the {\it angular component} of $x$. Moreover, let
$|x|_K:=|x|=q^{-{\rm ord}(x)}$ be its {\it absolute value},
and let $|dx|$ be the Haar measure on $K^n$ such that the
measure of $\mathcal{O}_K^{n}$ is one. Let
$f(x)\in \mathcal{O}_K[x]$ with $x=(x_1, \cdots, x_n)$ be a
non-constant polynomial and let
$\chi: \mathcal{O}_K^{\times}\rightarrow \mathbb{C}^{\times}$
be any given multiplicative character of $\mathcal{O}_K^{\times}$.
We put $\chi(0):=0$. Then for any $s\in\mathbb{C}$ with ${\rm Re}(s)>0$,
the {\it Igusa's local zeta function} $Z_f(s, \chi)$ is defined as
$$Z_f(s, \chi):=\int_{\mathcal{O}_K^{n}}{\chi(ac f(x))|f(x)|^s|dx|}.$$

When ${\rm char}(K)=0$, that is, $K$ is a finite extension of the
$p$-adic field $\mathbb{Q}_p$, Igusa \cite{[Igu1]} \cite{[Igu2]} proved
that $Z_f(s, \chi)$ is a rational function of $q^{-s}$ by using the
resolution of singularities. On the other hand, let
$$N_i:=\#\{x\in (\mathcal{O}_K/\mathcal{P}_K^i)^{n}
|f(x)\equiv 0 \pmod{\mathcal{P}_K^i}\},$$
and set $N_0:=1$. We denote the {\it Poincar$\acute{e}$ series} of $f$ by
$$P(t):=\sum_{i=0}^{\infty}N_i(q^{-n}t)^i.$$
Then (see, for example \cite{[M]}):
$$P(q^{-s})=\dfrac{1-q^{-s}Z_f(s, \chi_{{\rm triv}})}{1-q^{-s}}$$
with $\chi_{{\rm triv}}$ being the trivial character.
Thus, as a consequence of Igusa's rationality theorem, one
knows that $P(q^{-s})$ is a rational function of $q^{-s}$. A proof
of the rationality of $P(q^{-s})$ without using resolution of
singularities was given by Denef \cite{[De1]}. Actually, he used
a deep result for the $p$-adic fields what he called $p$-adic
cell decomposition. Furthermore, using this relation, Segers
\cite{[S]} gave the lower bound for the poles of Igusa's local
zeta functions.

However, when ${\rm char}(K)>0$, the question of rationality
becomes much more difficult because the techniques used by Igusa
and Denef are not available in this case. In particular, for any
$f(x)\in\mathcal{O}_K[x]$, the rationality of $Z_f(s, \chi)$
is still kept open. In fact, there are few known results up to now.
For example, Z$\acute{u}$$\tilde{n}$iga-Galindo \cite{[ZG1]} proved that
$Z_f(s, \chi_{{\rm triv}})$ is a rational function of $q^{-s}$
when $f$ is a semiquasihomogeneous polynomial with an absolutely
algebraically isolated singularity at the origin. The basic tool he used
is called the {\it $\pi$-adic stationary phase formula}, which was first
introduced by Igusa \cite{[Igu3]}. For more progress on the rationality
of Igusa local zeta function, the readers are referred to \cite{[M]}.

In this paper, we study the local zeta functions for a class of
hybrid polynomials for the case of positive characteristic. We
establish  the rationality of the local zeta functions for these
hybrid polynomials and also obtain all the candidate poles.
From now on, we let $K$ be a non-archimedean local field of
characteristic $p$ with $p$ being a fixed prime number. As usual,
for any $x\in\mathbb{R}$,
let $\lfloor x\rfloor$ be the largest integer which is no more than $x$.
First of all, we give the definition of hybrid polynomial in three variables.
\begin{defn}\label{defn 1.1}
Let $k$, $r$ and $l$ be positive integers with
$p\nmid rl$, $p\mid (r+l+k)$ and $\bar{r}+\bar{l}\leq p$, where $\bar{r}$ and
$\bar{l}$ denote the residues of $r$ and $l$ modulo $p$. If $t$ is an arbitrary
constant from the ground field $K$, then the polynomial
\begin{align}\label{1.1}
f(x, y, z)=&x^p+y^rz^l\sum_{i=0}^{k}{\dbinom{k+r}{i+r}y^i(tz-y)^{k-i}}\nonumber\\
:=&x^p+y^rz^l\mathbb{H}_r^k(y, tz-y)
\end{align}
is called a {\it hybrid polynomial}, where
$\mathbb{H}_r^k(y, w)=\sum\limits_{i=0}^{k}{\binom{k+r}{i+r}y^iw^{k-i}}$.
\end{defn}

Hybrid polynomials were first introduced by Hauser (see \cite{[Ha1]}
or \cite{[Ha2]}) to study the resolution of singularities in positive
characteristic. With $f(x,y,z)$ being a hybrid polynomial with $r=1$
and $t^{(k+1)/p}\in\mathcal{O}_K^{\times}$, Le$\acute{o}$n-Cardenal,
Ibadula and Segers \cite{[LIS]} proved that $Z_f(s, \chi)$ is a rational
function of $q^{-s}$ with only three candidate poles when $l=1$. Later on,
for any positive integer $l$, Yin and Hong \cite{[YH]} showed that
$Z_f(s, \chi)$ is a rational function of $q^{-s}$ with only four
possible poles of $Z_f(s, \chi)$.

In the present paper, we extend the works presented in \cite{[LIS]} and
\cite{[YH]} to the general case. In other words, we study the local
zeta function for arbitrary hybrid polynomial with only a natural
assumption on the constant $t$. Before considering the local zeta
function of hybrid polynomials, we first study a class of two-variable
polynomials of the following form:
\begin{equation}\label{eqno 1.2}
g(u,v):=\sum_{i=i_0}^{\infty}\alpha_iu^i
+\sum_{j=j_0}^{\infty}\beta_jv^j\in\mathcal{O}_K[u, v],
\end{equation}
where $i_0, j_0$ are positive integers with $\gcd(i_0, j_0)=1$,
and $\alpha_i, \beta_j\in\mathcal{O}_K$ satisfying that
$\alpha_{i_0}\beta_{j_0}\ne 0$ and for all integers $i$ and $j$ with
$i\ge i_0+1$, $j\ge j_0+1$, ${\rm ord}(\alpha_i)>{\rm ord}(\alpha_{i_0})$
and ${\rm ord}(\beta_j)>{\rm ord}(\beta_{j_0})$.
Notice that the last condition implies that
${\rm ord}(\alpha_i)\ge 1$ and ${\rm ord}(\beta_j)\ge 1$ since
${\rm ord}(\alpha_{i_0})\ge 0$ and ${\rm ord}(\beta_{j_0})\ge 0$,
i.e., $\pi\mid \alpha_i$ and $\pi\mid \beta_j$.
We study the local zeta function of $g$ and prove
its rationality. That is, we have the following result.

\begin{thm}\label{thm 1.1}
Let $g(u, v)$ be a polynomial of the form (\ref{eqno 1.2}).
Then the local zeta function $Z_g(s, \chi)$
is a rational function of $q^{-s}$. Moreover, we have
$$Z_g(s, \chi)=\dfrac{G(q^{-s})}{(1-q^{-1-s})(1-q^{-i_0-j_0-i_0j_0s})},$$
where $G(x)\in\mathbb{C}[x]$.

\end{thm}
Using Theorem \ref{thm 1.1}, we prove the rationality of local zeta
functions for the hybrid polynomial and provide all the candidate poles as follows.

\begin{thm}\label{thm 1.2}
Let $f(x, y, z)=x^p+y^rz^l\mathbb{H}_r^k(y, tz-y)$ be a hybrid polynomial with
$t^{1/p}\in \mathcal{O}_K^{\times}$. Then the local zeta function
$Z_f(s, \chi)$ is a rational function of $q^{-s}$. Furthermore, 
each of the following is true:

{\rm (1)} If $p\mid\binom{k+r}{r}$, then
\begin{align*}
Z_f(s, \chi)={\left\{\begin{array}{rl}
\dfrac{1-q^{-1}}{1-q^{-1-ps}}, \ \ \ \ &{\rm if} \ \chi^p=\chi_{{\rm triv}},\\
0,\ \ \ \ &{\rm otherwise}.
\end{array}\right.}
\end{align*}

{\rm (2)} If $p\nmid\binom{k+r}{r}$, then
$$
Z_f(s, \chi)=\dfrac{F(q^{-s})}{(1-q^{-1-s})(1-q^{-((k+l+r)/p+2)-(k+l+r)s})
\prod\limits_{a\in \{r, l, k+1\}}(1-q^{-p-a-pas})},
$$
where $F(x)\in\mathbb{C}[x]$.
\end{thm}
As a direct application of Theorem \ref{thm 1.2},
we list all the candidate poles of $Z_f(s, \chi)$.

\begin{thm}\label{thm 1.3}
Let $f(x, y, z)$ be a hybrid polynomial and $t^{1/p}\in \mathcal{O}_K^{\times}$.
Let $s$ be a pole of $Z_f(s, \chi)$. Then each of the following is true:

{\rm (1)} If $p\mid\binom{k+r}{r}$, then
$$s=-\dfrac{1}{p}+\dfrac{2\pi i\mathbb{Z}}{p\log q},$$

{\rm (2)} If $p\nmid\binom{k+r}{r}$,
then $s$ is equal to one of the following five forms:
\begin{align*}
&-1+\dfrac{2\pi i\mathbb{Z}}{\log q},\ \ \ \ \
-\dfrac{1}{p}-\dfrac{2}{k+r+l}+\dfrac{2\pi i\mathbb{Z}}{(k+r+l)\log q},\\
&-\dfrac{1}{p}-\dfrac{1}{r}+\dfrac{2\pi i\mathbb{Z}}{pr\log q},
\ \ \ -\dfrac{1}{p}-\dfrac{1}{l}+\dfrac{2\pi i\mathbb{Z}}{pl\log q},
\ \ \ -\dfrac{1}{p}-\dfrac{1}{k+1}+\dfrac{2\pi i\mathbb{Z}}{p(k+1)\log q}.
\end{align*}
\end{thm}

\noindent{\bf Remark 1.5.}
(1) Notice that the condition $t^{1/p}\in \mathcal{O}_K^{\times}$
in Theorem \ref{thm 1.2} is slight different from the assumption
given in \cite{[LIS]} and \cite{[YH]} that states
$t^{(k+1)/p}\in \mathcal{O}_K^{\times}$ with $p\nmid (k+1)$.
But it is easy to see that these two assumptions are equivalent.
Let us show it here. On the one hand, it is obvious that if
$t^{1/p}\in \mathcal{O}_K^{\times}$, then
$t^{(k+1)/p}\in \mathcal{O}_K^{\times}$. On the other hand,
let $t^{(k+1)/p}\in \mathcal{O}_K^{\times}$. Since
$p\nmid(k+1)$, one has $\gcd(p, k+1)=1$. It follows that
there exist integers $u$ and $v$ such that $u(k+1)+vp=1$. Hence
$$
t^{1/p}=t^{(u(k+1)+vp)/p}=(t^{(k+1)/p})^u t^v\in K.
$$
But $|t^{1/p}|^{k+1}=|t^{(k+1)/p}|=1$. So we have $|t^{1/p}|=1$
which implies that $t^{1/p}\in \mathcal{O}_K^{\times}$.

(2) If $r=1$, one has $\binom{k+r}{r}=k+1$. It then follows that
$p\nmid (k+1)$ (see (1)). Thus by Theorem \ref{thm 1.2} (2),
we obtain the same four poles given in \cite{[YH]} in the case $r=1$,
but we also get a new extra pole, which depends on the characteristic
of $K$ and the parameter $r$ when $r>2$.
We will explain later why this pole disappears if $r=1$.

The paper is organized as follows. In Section 2, we first supply
more information about the hybrid polynomials. Then we review the
stationary phase formula and some results about Newton polyhedra.
Furthermore, we prove some lemmas which are needed in the proof of
our main theorems. Consequently, in Sections 3 and 4, we give the
proofs of Theorems \ref{thm 1.1} and \ref{thm 1.2}, respectively.
Finally, we present some examples to demonstrate Theorems
\ref{thm 1.1} and \ref{thm 1.2}.

\section{\bf Preliminaries and lemmas }

\subsection{\bf Hybrid polynomial}

In Section 1, we introduced the hybrid polynomial, that is,
$$f(x, y, z)=x^p+y^rz^l\mathbb{H}_r^k(y, tz-y),$$
where
$$\mathbb{H}_r^k(y, w)=\sum\limits_{i=0}^{k}{\binom{k+r}{i+r}y^iw^{k-i}}.$$
In this section, we list some results about $\mathbb{H}_r^k(y, w)$.
One can find more details in the original paper of Hauser \cite{[Ha1]}.

First, it follows immediately from the definition that $\mathbb{H}_r^k(y, tz-y)$
is a homogeneous polynomial of degree $k$, i.e., for any $a\in \mathcal{O}_K$, we have
$$\mathbb{H}_r^k(ay, taz-ay)=a^k\mathbb{H}_r^k(y, tz-y).$$ 
Moreover, we have the following identity:
\begin{equation}\label{eqno 2.1}
\mathbb{H}_r^k(y, tz-y)=\sum_{i=0}^k(-1)^i\dbinom{k+r}{k-i}
\dbinom{i+r-1}{i}t^{k-i}y^iz^{k-i}.
\end{equation}
For a proof of (\ref{eqno 2.1}), we refer the readers to \cite{[Ha1]}.
Furthermore, we have the following result.
\begin{lem}\label{lem 2.1}
We have
\begin{equation*}
\dfrac{\partial}{\partial y}(y^rz^l\mathbb{H}_r^k(y, tz-y))
=r\dbinom{k+r}{r}y^{r-1}z^l(tz-y)^k.
\end{equation*}
\end{lem}

\begin{proof}
The proof is straightforward. First, we deduce that
\begin{align*}
y^rz^l\mathbb{H}_r^k(y, tz-y)
=&z^l\sum_{i=0}^{k}{\dbinom{k+r}{i+r}y^{i+r}(tz-y)^{k-i}}\\
=&z^l\sum_{i=r}^{k+r}{\dbinom{k+r}{i}y^i(tz-y)^{k+r-i}}.
\end{align*}
It then follows that
\begin{align}\label{2.2}
&\dfrac{\partial}{\partial y}(y^rz^l\mathbb{H}_r^k(y, tz-y))\nonumber\\
=&z^l\sum_{i=r}^{k+r}{\dbinom{k+r}{i}iy^{i-1}(tz-y)^{k+r-i}}
-z^l\sum_{i=r}^{k+r-1}{\dbinom{k+r}{i}(k+r-i)y^i(tz-y)^{k+r-i-1}}\nonumber\\
=&z^l\Big(\sum_{i=r-1}^{k+r-1}{\dbinom{k+r}{i+1}(i+1)y^i(tz-y)^{k+r-i-1}}
-\sum_{i=r}^{k+r-1}{\dbinom{k+r}{i}(k+r-i)y^i(tz-y)^{k+r-i-1}}\Big)\nonumber\\
=&r\dbinom{k+r}{r}y^{r-1}z^l(tz-y)^k\nonumber\\
+&z^l\sum_{i=r}^{k+r-1}{\Big(\dbinom{k+r}{i+1}(i+1)
-\dbinom{k+r}{i}(k+r-i)\Big)y^i(tz-y)^{k+r-i-1}}
\end{align}
On the other hand, for any integer $i$ with $r\le i\le k+r-1$, one has
\begin{align}\label{2.3}
\dbinom{k+r}{i+1}(i+1)=&\dfrac{(k+r)\cdots(k+r-i)}{(i+1)!}(i+1)\nonumber\\
=&\dfrac{(k+r)\cdots(k+r-i+1)}{i!}(k+r-i)\nonumber\\
=&\dbinom{k+r}{i}(k+r-i)
\end{align}
So from (\ref{2.2}) and (\ref{2.3}), we arrive at
$$\dfrac{\partial}{\partial y}(y^rz^l\mathbb{H}_r^k(y, tz-y))
=r\dbinom{k+r}{r}y^{r-1}z^l(tz-y)^k$$
as desired. So Lemma \ref{lem 2.1} is proved.
\end{proof}

\subsection{\bf Stationary phase formula}

In \cite{[Igu3]}, Igusa first introduced the stationary phase formula.
Since then it became a powerful tool to compute the local zeta function
in arbitrary characteristic. In this section, we recall the stationary 
phase formula.

For any $x\in\mathcal{O}_K^n$, let $\bar{x}$ be the image of $x$ under
the canonical homomorphism
$$\mathcal{O}_K^n\rightarrow (\mathcal{O}_K/\pi\mathcal{O}_K)^n\cong\mathbb{F}_q^n.$$
For $f(x)\in \mathcal{O}_K[x]$, $\bar{f} (x)$ stands for the polynomial
obtained by reducing modulo $\pi$ the coefficients of $f(x)$.
Let $A$ be any ring and $f(x)\in A[x]$. We define
$V_f(A):=\{x\in A^n|f(x)=0\}$. Let Sing$_f(A)$ be the set
of $A$-value singular points of $V_f$, namely,
$${\rm Sing}_f(A):=\Big\{x\in A^n\Big|f(x)
=\dfrac{\partial f}{\partial x_1}(x)=\cdots
=\dfrac{\partial f}{\partial x_n}(x)=0\Big\}.$$

We fix a lifting $R$ of $\mathbb{F}_q$ in $\mathcal{O}_K$. That is,
the set $R^n$ is mapped bijectively onto $\mathbb{F}_q^n$ by the
canonical homomorphism. Let $\bar{D}$ be a subset of 
$\mathbb{F}_q^n$ and let $D$ be its preimage
under the canonical homomorphism. Denote by $S(f, D)$ the subset
of $R^n$ mapped bijectively to the set
${\rm Sing}_{\bar{f}}(\mathbb{F}_q)\cap \bar{D}$. Furthermore, we denote
\begin{align*}
v(\bar{f}, D, \chi):={\left\{\begin{array}{rl}
q^{-n} \cdot \#\{\bar{P}\in \bar{D}|\bar{P}\notin V_{\bar{f}}(\mathbb{F}_q)\}, \ \ \ \
&{\rm if} \ \chi=\chi_{{\rm triv}},\\
q^{-nc_{\chi}}\sum\limits_{\{P\in D|\bar{P}\notin V_{\bar{f}}(\mathbb{F}_q)\}
{\rm mod}\ \mathcal{P}_K^{c_{\chi}}}{\chi(ac f(P)),}\  \ \ \
&{\rm if} \ \chi \neq \chi_{{\rm triv}},
\end{array}\right.}
\end{align*}
where $c_{\chi}$ is the conductor of $\chi$, and
\begin{align*}
\sigma(\bar{f}, D, \chi):={\left\{\begin{array}{rl}
q^{-n}\cdot \#\{\bar{P}\in \bar{D}|\bar{P}\ {\rm is\ a\ nonsingular\ point\ of} \
V_{\bar{f}}(\mathbb{F}_q)\}, \ \  &{\rm if} \ \chi=\chi_{{\rm triv},}\\
0,\ \   & {\rm if} \ \chi \neq \chi_{{\rm triv}}.
\end{array}\right.}
\end{align*}
Finally, let
$$Z_f(s, \chi, D):=\int_D\chi(acf(x))|f(x)|^s|dx|.$$

Now we can state the Igusa's stationary phase formula in
the following form.

\begin{lem}{\rm (Stationary phase formula)} {\rm \cite{[YH]}}\label{lem 2.2}
For any complex number $s$ with ${\rm Re}(s)>0$, we have
\begin{align*}
Z_f(s, \chi, D)=v(\bar{f}, D, \chi)&+\sigma(\bar{f}, D, \chi)
\dfrac{(1-q^{-1})q^{-s}}{1-q^{-1-s}}+Z_f(s, \chi, D_{S(f, D)}),
\end{align*}
where 
$$D_{S(f, D)}:=\bigcup\limits_{P\in S(f, D)}D_P$$ 
with
$D_P:=\{x\in\mathcal{O}_K^n|x-P\in(\pi\mathcal{O}_K)^n\}$,
that is, $D_{S(f, D)}$ is the preimage of
${\rm Sing}_{\bar{f}}(\mathbb{F}_q)\cap \bar{D}$
under the canonical homomorphism
$\mathcal{O}_K^n\rightarrow (\mathcal{O}_K/\pi\mathcal{O}_K)^n$.
\end{lem}

\subsection{\bf Newton polyhedra}
In this section, we review some results about Newton polyhedra.
One can see more details in \cite{[ZG2]}.

First, we define a set as $\mathbb{R}_{+}:=\{x\in \mathbb{R}|x\geq 0\}$. Let
$f(x)=\sum_{l}{a_lx^l}\in K[x]$ be a polynomial in $n$ variables satisfying
$f(0)=0$, where the notation
$a_lx^l=a_{l_1, \cdots l_n}{x_1}^{l_1}\cdots {x_n}^{l_n}$, $l=(l_1, \cdots, l_n)$.
The {\it support set} of $f$, denoted by ${\rm supp}(f)$, is defined by 
$${\rm supp}(f)=\{l\in \mathbb{N}^n|a_l\neq 0\}.$$ 
Then we define the {\it Newton polyhedra} $\Gamma(f)$
of $f$ as the convex hull in $\mathbb{R}_{+}^n$ of the set
$$\bigcup_{l\in {\rm supp}(f)}{(l+\mathbb{R}_{+}^n)}.$$

By a proper {\it face} $\gamma$ of $\Gamma(f)$, we mean the non-empty convex set
$\gamma$ obtained by intersecting $\Gamma(f)$ with an affine hyperplane $H$,
such that $\Gamma(f)$ is contained in one of two half-plane determined by $H$.
The hyperplane $H$ is called the {\it supporting hyperplane} of $\gamma$. Let
$a_{\gamma}=(a_1, a_2,\cdots, a_n)\in \mathbb{N}^n\backslash\{0\}$
denote the vector that is perpendicular with the supporting
hyperplane $H$ and let  $|a_{\gamma}|:=\sum\limits_{i}{a_i}$.
A face of codimension one is named {\it facet}.

Let $\langle , \rangle$ denote the usual inner product of $\mathbb{R}^n$.
For any $a\in \mathbb{R}^n_{+}$, we define
$$m(a):=\inf_{b\in \Gamma(f)}{\{\langle a, b\rangle \}},$$
and for any $a\in \mathbb{R}^n_{+}\setminus \{0\}$,
{\it the first meet locus} of $a$ is defined as
$$F(a):=\{x\in \Gamma(f)|\langle a, x\rangle=m(a)\}.$$
In fact, $F(a)$ is a proper face of $\Gamma(f)$.

Now we define an equivalence relation $\simeq$ on $\mathbb{R}^n_{+}\setminus \{0\}$
as follows: For $a, a^{'}\in \mathbb{R}^n_{+}\setminus \{0\}$,
$$a\simeq a^{'} \ {\rm if\ and\ only\ if}\ F(a)=F(a^{'}).$$
Let $\gamma$ be a proper face of $\Gamma(f)$,
we define {\it the cone associated to $\gamma$} as
$$\Delta_{\gamma}:=\{a\in \mathbb{R}^n_{+}\setminus \{0\}|F(a)=\gamma\}.$$
From the definition, it immediately follows that
$\Delta_{\gamma}\cap\Delta_{\gamma^{'}}=\emptyset$
for different proper faces $\gamma, \gamma^{'}$ of $\Gamma(f)$.
The following lemma describe the generators of $\Delta_{\gamma}$,
and its proof can be found in \cite{[De2]}.

\begin{lem}\label{lem 2.3}
Let $\gamma$ be a proper face of $\Gamma(f)$ and let $\omega_1, \cdots, \omega_e$
be the facets of $\Gamma(f)$ which contain $\gamma$.
Let $\alpha_1, \cdots, \alpha_e$ be vectors which are perpendicular to 
$\omega_1, \cdots, \omega_e$, respectively. Then
\begin{equation*}
\Delta_{\gamma}=\Big\{\sum_{i=1}^ea_i\alpha_i|a_i\in \mathbb{R}^{+}\Big\}.
\end{equation*}
\end{lem}

From the above discussion, it follows that $\mathbb{R}^n_{+}$ can
be partitioned into equivalence classes module $\simeq$, that is
\begin{equation*}
\mathbb{R}^n_{+}=\{0\}\bigcup\bigcup_{\gamma}\Delta_{\gamma},
\end{equation*}
where $\gamma$ runs over all proper faces of $\Gamma(f)$.
Then one derives that
\begin{equation}\label{eqno 2.4}
\mathbb{N}^n=\{0\}\bigcup\bigcup_{\gamma}
\big(\Delta_{\gamma}\bigcap (\mathbb{N}^n\setminus \{0\})\big).
\end{equation}

Let $g(u, v)$ be a polynomial of the form (\ref{eqno 1.2}). Using Lemma 
\ref{lem 2.3}, we can calculate all the proper faces of $\Gamma(g)$,
and the cones associated to them. Explicitly, we have the following lemma. 

\begin{lem}\label{lem 2.4}
Let $g(u,v)=\sum\limits_{i=i_0}^{\infty}\alpha_iu^i
+\sum\limits_{j=j_0}^{\infty}\beta_jv^j\in\mathcal{O}_K[u, v]$.
Then one has
\begin{equation*}
\Gamma(g)=\{(u, v)|u\ge i_0, v\ge j_0, j_0u+i_0v\ge i_0j_0\}.
\end{equation*}
Moreover, $\Gamma(g)$ has exact five proper faces, and
\begin{align*}
\Delta_{\gamma}={\left\{\begin{array}{rl}
\{(0, a)|a\in \mathbb{R}^{+}\},\  &\gamma=\{(u, 0)|u\ge i_0\},\\
\{(bj_0, a+bi_0)|a, b\in \mathbb{R}^{+}\},\   &\gamma=(i_0, 0),\\
\{(aj_0, ai_0)|a\in \mathbb{R}^{+}\},\
&\gamma=\{(u, v)|j_0u+i_0v=i_0j_0, 0\le u\le i_0, 0\le v\le j_0\},\\
\{(aj_0+b, ai_0)|a, b\in \mathbb{R}^{+}\},\    &\gamma=(0, j_0),\\
\{(a, 0)|a\in \mathbb{R}^{+}\},\  &\gamma=\{(0, v)|v\ge j_0\}.
\end{array}\right.}
\end{align*}
\end{lem}

\begin{proof}

First of all, it is obvious that $\Gamma(g)$ has the five proper faces:
\begin{align*}
\gamma_1&=\{(u, 0)|u\ge i_0\},\\
\gamma_2&=(i_0, 0),\\
\gamma_3&=\{(u, v)|j_0u+i_0v=i_0j_0, 0\le u\le i_0, 0\le v\le j_0\},\\
\gamma_4&=(0, j_0),\\
\gamma_5&=\{(0, v)|v\ge j_0\}
\end{align*}
as desired.

For $\gamma_1$, we choose the vector which is perpendicular to
$\gamma_1$ to be $(0, 1)$. Then by Lemma \ref{lem 2.3}, we have
$\Delta_{\gamma_1}=\{a(0, 1)|a\in \mathbb{R}^{+}\}=\{(0, a)|a\in \mathbb{R}^{+}\}$.

For $\gamma_2$, it is easy to see the point is intersected by two lines
(actually two facets of $\Gamma(g)$) $\gamma_1$ and $\gamma_3$.
We choose the vectors which are perpendicular to
$\gamma_1$ and $\gamma_3$ to be $(0, 1)$ and $(j_0, i_0)$, respectively.
Then Lemma \ref{lem 2.3} gives us that
$$\Delta_{\gamma_2}=\{a(0, 1)+b(j_0, i_0)|a, b\in \mathbb{R}^{+}\}
=\{(bj_0, a+bi_0)|a, b\in \mathbb{R}^{+}\}.$$

For $\gamma_3$, we choose $(j_0, i_0)$ to be the vector
which is perpendicular to $\gamma_3$. Then using Lemma \ref{lem 2.3},
one derives that
$\Delta_{\gamma_3}=\{a(j_0, i_0)|a\in \mathbb{R}^{+}\}
=\{(aj_0, ai_0)|a\in \mathbb{R}^{+}\}.$

For $\gamma_4$, the point is intersected by  $\gamma_3$ and $\gamma_5$.
We choose the vectors which are perpendicular to $\gamma_3$ and $\gamma_5$
to be $(j_0, i_0)$ and $(1, 0)$, respectively. By Lemma \ref{lem 2.3},
we have
$$\Delta_{\gamma_4}=\{a(j_0, i_0)+b(1, 0)|a, b\in \mathbb{R}^{+}\}
=\{(aj_0+b, ai_0)|a, b\in \mathbb{R}^{+}\}.$$

For $\gamma_5$, we choose the vector which is perpendicular to
$\gamma_5$ to be $(1, 0)$. Hence Lemma \ref{lem 2.3} tells us that
$\Delta_{\gamma_5}=\{a(1, 0)|a\in \mathbb{R}^{+}\}=\{(a, 0)|a\in \mathbb{R}^{+}\}$.

This ends the proof of Lemma \ref{lem 2.4}.
\end{proof}

As an application of Lemma \ref{lem 2.4}, we have the following important lemma.

\begin{lem}\label{lem 2.5}
Let $g(u,v)=\sum\limits_{i=i_0}^{\infty}\alpha_iu^i
+\sum\limits_{j=j_0}^{\infty}\beta_jv^j\in\mathcal{O}_K[u, v]$ with
$\gamma_1, \cdots, \gamma_5$ being the proper faces of $\Gamma(g)$
defined in the proof of Lemma \ref{lem 2.4}. Then 
\begin{align*}
\Delta_{\gamma_c}\bigcap (\mathbb{N}^2\setminus \{0\})
={\left\{\begin{array}{rl}
\{(0, a)|a\in \mathbb{Z}^{+}\},\ \  &c=1\\
\bigcup\limits_{m=0}^{j_0-1}\big\{(m+bj_0, \frac{mi_0+n_m}{j_0}+bi_0+a)
|a, b\in \mathbb{Z}^{+}\},\ \ &c=2,\\
\{(aj_0, ai_0)|a\in \mathbb{Z}^{+}\},\ \ &c=3\\
\bigcup\limits_{m=0}^{i_0-1}\big\{(\frac{mj_0+n_m^{'}}{i_0}+aj_0+b, m+ai_0)
|a, b\in \mathbb{Z}^{+}\big\},\ \ &c=4,\\
\{(a, 0)|a\in \mathbb{Z}^{+}\},\ \  &c=5,
\end{array}\right.}
\end{align*}
where $n_m$ and $n_m^{'}$ are positive integers depending on $m$.
\end{lem}

\begin{proof}
For $c=1$ and $c=5$, the two cases are immediately followed
from Lemma \ref{lem 2.4}.

For $c=3$, Lemma \ref{lem 2.4} tells us that
$\Delta_{\gamma_3}=\{(aj_0, ai_0)|a\in \mathbb{R}^{+}\}$.
Since $\gcd(i_0, j_0)=1$, so there exists integers $s$ and $t$
such that $sj_0+ti_0=1$. Thus if
$(a_0j_0, a_0i_0)\in \Delta_{\gamma_3}\bigcap \mathbb{N}^2\setminus \{0\}$,
then
$$a_0=a_0(sj_0+ti_0)=sa_0j_0+ta_0i_0\in\mathbb{Z}.$$
Hence $a_0\in \mathbb{Z}\cap \mathbb{R}^{+}=\mathbb{Z}^{+}$.

On the contrary, it is obvious that if $a_0\in\mathbb{Z}^{+}$, then
$(a_0j_0, a_0i_0)\in \Delta_{\gamma_3}\bigcap (\mathbb{N}^2\setminus \{0\})$.
Therefore we arrive at
$$\Delta_{\gamma_3}\bigcap(\mathbb{N}^2\setminus \{0\})
=\{(aj_0, ai_0)|a\in \mathbb{Z}^{+}\}$$
as desired.

For $c=2$ and $c=4$, it is easy to see that the two cases has symmetry.
Thus in what follows, we only give the proof of the case $c=4$.

By Lemma \ref{lem 2.4}, we have
$\Delta_{\gamma_4}=\{(aj_0+b, ai_0)|a, b\in \mathbb{R}^{+}\}$.
First one easily derives that if $a, b\in \mathbb{Z}^{+}$,
then $(aj_0+b, ai_0)\in \Delta_{\gamma_4}\bigcap (\mathbb{N}^2\setminus \{0\})$.
To prove the contrary, we consider the following two cases.

{\sc Case 1.} $i_0=1$. Then if
$(a_0j_0+b_0, a_0i_0)\in \Delta_{\gamma_4}\bigcap(\mathbb{N}^2\setminus \{0\})$,
it infers that $a_0=a_0i_0\in \mathbb{Z}\cap \mathbb{R}^{+}=\mathbb{Z}^{+}$.
Thus $a_0j_0+b_0\in \mathbb{Z}$
tells us that $b_0\in \mathbb{Z}\cap \mathbb{R}^{+}=\mathbb{Z}^{+}$.
Hence we have
$$\Delta_{\gamma_4}\bigcap (\mathbb{N}^2\setminus \{0\})=
\{(aj_0+b, ai_0)|a, b\in \mathbb{Z}^{+}\}.$$

{\sc Case 2.} $i_0>1$.
First we notice that if $a, b\notin \mathbb{Z}^{+}$,
then we only need to consider the case $0<a, b<1$.
In fact, let $a=\lfloor a\rfloor+a^{'}$ and
$b=\lfloor b\rfloor+b^{'}$, where $0<a^{'}, b^{'}<1$. Then
$(aj_0+b, ai_0)\in \Delta_{\gamma_4}\bigcap (\mathbb{N}^2\setminus \{0\})$
gives us that $a^{'}i_0=ai_0-\lfloor a\rfloor i_0\in \mathbb{Z}^{+}$,
and $a^{'}j_0+b^{'}=aj_0+b-(\lfloor a\rfloor j_0+\lfloor b\rfloor)
\in \mathbb{Z}^{+}$.

Now suppose
$(a_0j_0+b_0, a_0i_0)\in \Delta_{\gamma_4}\bigcap (\mathbb{N}^2\setminus \{0\})$
with $0<a_0, b_0<1$. Since $a_0i_0\in \mathbb{Z}^{+}$, we deduce that
\begin{equation}\label{eqno 2.5}
a_0\in \Big\{\frac{1}{i_0}, \frac{2}{i_0}, \cdots, \frac{i_0-1}{i_0}\Big\}.
\end{equation}
Let $a_0=\frac{m}{i_0}$ with $1\le m\le i_0-1$. Then
$a_0j_0+b_0\in \mathbb{Z}$ gives us that
$$b_0+\Big(\dfrac{j_0m}{i_0}-\Big\lfloor\dfrac{j_0m}{i_0}\Big\rfloor\Big)\in \mathbb{Z}.$$
Because $\gcd(i_0, j_0)=1$, so $\frac{j_0m}{i_0}\notin \mathbb{Z}$,
which implies that
$0<\frac{j_0m}{i_0}-\lfloor\frac{j_0m}{i_0}\rfloor<1$.
Thus one has
\begin{equation}\label{eqno 2.6}
b_0=1-\Big(\dfrac{j_0m}{i_0}-\Big\lfloor\dfrac{j_0m}{i_0}\Big\rfloor\Big)
:=\dfrac{n_m^{'}}{i_0},
\end{equation}
where $n_m^{'}$ is an integer depended on $m$ and satisfies $1\le n_m^{'}\le i_0-1$.
By (\ref{eqno 2.5}) and (\ref{eqno 2.6}), it follows that the set
$\{(aj_0+b, ai_0)|0<a, b<1\}$ contains exact $i_0-1$ points
with positive integer coordinates:
$\big(\frac{j_0+n_1^{'}}{i_0}, 1\big)$, $\cdots$,
$\big(\frac{(i_0-1)j_0+n_{i_0-1}^{'}}{i_0}, i_0-1\big)$.
Let $n_0^{'}=0$. Then we have that
\begin{align}\label{2.7}
\Delta_{\gamma_4}\bigcap (\mathbb{N}^2\setminus \{0\})
=&\bigcup_{m=0}^{i_0-1}\Big\{\big(\frac{mj_0+n_m^{'}}{i_0}, m\big)+\big(aj_0+b, ai_0\big)
|a, b\in \mathbb{Z}^{+}\Big\}\nonumber\\
=&\bigcup_{m=0}^{i_0-1}\Big\{\big(\frac{mj_0+n_m^{'}}{i_0}+aj_0+b, m+ai_0\big)
|a, b\in \mathbb{Z}^{+}\Big\}.
\end{align}
Hence this case is true. Notice that if $i_0=1$, then result in {\sc Case 1}
tells us that (\ref{2.7}) still holds. This finishes the proof of the case $c=4$.

Using the exact same discussion, we can obtain that
\begin{align*}
\Delta_{\gamma_2}\bigcap (\mathbb{N}^2\setminus \{0\})=
\bigcup\limits_{m=0}^{j_0-1}\Big\{(m+bj_0, \frac{mi_0+n_m}{j_0}+bi_0+a)
|a, b\in \mathbb{Z}^{+}\Big\}
\end{align*}
where $n_m$ is an integer depended on $m$ with $1\le n_m\le j_0-1$
and $n_0=0$.

This concludes the proof of Lemma \ref{lem 2.5}.
\end{proof}

\subsection{\bf Some lemmas}
In this section, we show some lemmas which will be used in
the proof of our main theorems. First we give a definition as follows.

\begin{defn}
Let $f(x)\in\mathcal{O}_K[x]$ be a polynomial and
$P=(a_1, \cdots, a_n)\in \mathcal{O}_K^n$, such that
$P\notin {\rm Sing}_f(\mathcal{O}_K)$. We define
$$L(f, P):=\min\Big({\rm ord}(f(P)), {\rm ord}(\dfrac{\partial f}{\partial x_1}(P)),
\cdots, {\rm ord}(\dfrac{\partial f}{\partial x_n}(P))\Big).$$
\end{defn}

The index $L(f, P)$ was first introduced by
Z$\acute{u}$$\tilde{n}$iga-Galindo (see \cite{[ZG1]}) to study the singularity of $f$
at a point $P$ with $P\notin {\rm Sing}_f(\mathcal{O}_K)$. In \cite{[ZG2]},
he proved that for any polynomial $f(x)\in\mathcal{O}_K[x]$ such that
${\rm Sing}_f(\mathcal{O}_K)\cap(\mathcal{O}_K^{\times})^n=\emptyset$,
$L(f, P)$ is bounded by a constant only depending on $f$ and $(\mathcal{O}_K^{\times})^n$
for all $P\in(\mathcal{O}_K^{\times})^n$. Actually, he proved a more general
result, but we only need this special case (see \cite{[ZG2]}).
Let $C(f, (\mathcal{O}_K^{\times})^n)\in \mathbb{N}$ be the minimal constant such that
$L(f, P)\le C(f, (\mathcal{O}_K^{\times})^n)$ for all $P\in(\mathcal{O}_K^{\times})^n$.
Now we can state the following lemma.

\begin{lem}\label{lem 2.7}{\rm (Corollary 2.5, \cite{[ZG2]})}
Let $F(x)=f(x)+\pi^\beta g(x)\in\mathcal{O}_K[x]$ be a polynomial such that
$\beta\ge C(f, (\mathcal{O}_K^{\times})^n)+1$, and
$${\rm Sing}_F(\mathcal{O}_K)\bigcap(\mathcal{O}_K^{\times})^n
={\rm Sing}_f(\mathcal{O}_K)\bigcap(\mathcal{O}_K^{\times})^n=\emptyset.$$
Then
$$Z_F(s, \chi, (\mathcal{O}_K^{\times})^n)=Z_f(s, \chi, (\mathcal{O}_K^{\times})^n).$$
\end{lem}

Using Lemma \ref{lem 2.7}, we can prove the following useful lemma,
which will play a key role in the proof of Theorem \ref{thm 1.1}.

\begin{lem}\label{lem 2.8}
Let
$$g(u, v)=\sum\limits_{i=i_0}^{\infty}\alpha_iu^i
+\sum\limits_{j=j_0}^{\infty}\beta_jv^j$$
be a polynomial of the form (\ref{eqno 1.2}), and let
$e_0:=\min({\rm ord}(\alpha_{i_0}), {\rm ord}(\beta_{j_0}))$
and $\tilde{g}(u, v):=\alpha_{i_0}u^{i_0}+\beta_{j_0}v^{j_0}$.
Then
\begin{align*}
Z_g(s, \chi, (\mathcal{O}_K^{\times})^2)
=Z_{\tilde{g}}(s, \chi, (\mathcal{O}_K^{\times})^2)
={\left\{\begin{array}{rl}
\dfrac{G_0(q^{-s})}{1-q^{-1-s}},\ \
&{\rm ord}(\alpha_{i_0})={\rm ord}(\beta_{j_0}),\\
c(\chi)q^{-e_0s},\ \ &{\rm otherwise},
\end{array}\right.}
\end{align*}
where $G_0(x)$ is a polynomial with complex coefficients depending on $\chi$,
and $c(\chi)$ is a constant depending on $\chi$ such that $|c(\chi)|_{\infty}\le 1$
with $|\cdot|_{\infty}$ being the usual absolute value on $\mathbb{C}$.
\end{lem}

\begin{proof}
Since $\gcd(i_0, j_0)=1$, it then follows that either $p\nmid i_0$ or $p\nmid j_0$.
Without loss of generality, we can assume $p\nmid i_0$.
Then for any $\tilde{P}=(\tilde{u}_0, \tilde{v}_0)\in (\mathcal{O}_K^{\times})^2$,
we have
$\frac{\partial \tilde{g}}{\partial u}(\tilde{P})=\alpha_{i_0}i_0\tilde{u}_0^{i_0-1}\ne 0$,
which implies that
$${\rm Sing}_{\tilde{g}}(\mathcal{O}_K)\bigcap(\mathcal{O}_K^{\times})^2=\emptyset.$$

Moreover, one has
${\rm ord}(\frac{\partial \tilde{g}}{\partial u}(\tilde{P}))
={\rm ord}(\alpha_{i_0}i_0\tilde{u}_0^{i_0-1})={\rm ord}(\alpha_{i_0})$
because $i_0\in\mathcal{O}_K^{\times}$.
If ${\rm ord}(\alpha_{i_0})\le {\rm ord}(\beta_{j_0})$, then we arrive at
$$L(\tilde{g}, \tilde{P})\le {\rm ord}(\alpha_{i_0})=e_0.$$
If ${\rm ord}(\alpha_{i_0})>{\rm ord}(\beta_{j_0})$, then we have
${\rm ord}(\beta_{j_0})=e_0$ and
$${\rm ord}(\tilde{g}(\tilde{P}))={\rm ord}(\pi^{e_0}
(\pi^{-e_0}\alpha_{i_0}\tilde{u}_{0}^{i_0}
+\pi^{-e_0}\beta_{j_0}\tilde{v}_0^{j_0}))=e_0$$
since ${\rm ord}(\pi^{-e_0}\alpha_{i_0}\tilde{u}_{0}^{i_0}
+\pi^{-e_0}\beta_{j_0}\tilde{v}_0^{j_0})=0$.
Hence we still have
$$L(\tilde{g}, \tilde{P})\le e_0.$$

On the other hand, we get that
$\frac{\partial g}{\partial u}=\sum\limits_{i=i_0}^{\infty}i\alpha_iu^{i-1}$.
Then for any $P=(u_0, v_0)\in (\mathcal{O}_K^{\times})^2$, we deduce that
\begin{align*}
\frac{\partial g}{\partial u}(P)=&\sum\limits_{i=i_0}^{\infty}i\alpha_iu_0^{i-1}
=\pi^{{\rm ord}(\alpha_{i_0})}
\sum\limits_{i=i_0}^{\infty}i\pi^{-{\rm ord}(\alpha_{i_0})}\alpha_iu_0^{i-1}\\
=&\pi^{{\rm ord}(\alpha_{i_0})}(\pi^{-{\rm ord}(\alpha_{i_0})}\alpha_{i_0}+
\sum\limits_{i=i_0+1}^{\infty}i\pi^{-{\rm ord}(\alpha_{i_0})}\alpha_iu_0^{i-1}).
\end{align*}
Because for all $i>i_0$, ${\rm ord}(i)>{\rm ord}(i_0)$. Thus
$${\rm ord}(\pi^{-{\rm ord}(i_0)}\alpha_i)={\rm ord}(\alpha_i)-{\rm ord}(\alpha_{i_0})>0.$$
But ${\rm ord}(\pi^{-{\rm ord}(\alpha_{i_0})}\alpha_{i_0})=0$. Hence it follows that
\begin{align*}
{\rm ord}(\frac{\partial g}{\partial u}(P))
=&{\rm ord}\Big(\pi^{{\rm ord}(\alpha_{i_0})}(\pi^{-{\rm ord}(\alpha_{i_0})}\alpha_{i_0}
+\sum\limits_{i=i_0+1}^{\infty}i\pi^{-{\rm ord}(\alpha_{i_0})}\alpha_iu_0^{i-1})\Big)\\
=&{\rm ord}(\pi^{{\rm ord}(\alpha_{i_0})})+{\rm ord}(\pi^{-{\rm ord}(\alpha_{i_0})}\alpha_{i_0}
+\sum\limits_{i=i_0+1}^{\infty}i\pi^{-{\rm ord}(\alpha_{i_0})}\alpha_iu_0^{i-1})\\
=&{\rm ord}(\alpha_{i_0}).
\end{align*}
Then one deduces that $\frac{\partial g}{\partial u}(P)\ne 0$, which implies that
$${\rm Sing}_g(\mathcal{O}_K)\bigcap(\mathcal{O}_K^{\times})^2=\emptyset.$$
Moreover, since ${\rm ord}(i)>{\rm ord}(i_0)$ and ${\rm ord}(j)>{\rm ord}(j_0)$
for all $i>i_0$, $j>j_0$, then we have
\begin{align*}
g(u, v)=&\tilde{g}(u, v)+\sum\limits_{i=i_0+1}^{\infty}\alpha_iu^i
+\sum\limits_{j=j_0+1}^{\infty}\beta_jv^j\\
=&\tilde{g}(u, v)+\pi^{e_0+1}\Big(\sum\limits_{i=i_0+1}^{\infty}\pi^{-(e_0+1)}\alpha_iu^i
+\sum\limits_{j=j_0+1}^{\infty}\pi^{-(e_0+1)}\beta_jv^j\Big),
\end{align*}
where $\pi^{-(e_0+1)}\alpha_i, \pi^{-(e_0+1)}\beta_j\in \mathcal{O}_K$.
Therefore by Lemma \ref{lem 2.7}, we arrive at
$$Z_g(s, \chi, (\mathcal{O}_K^{\times})^2)
=Z_{\tilde{g}}(s, \chi, (\mathcal{O}_K^{\times})^2)$$
as expected. So the proof of the first equality is complete.

Now we prove the second equality by considering the following two cases.

{\sc Case 1.} ${\rm ord}(\alpha_{i_0})={\rm ord}(\beta_{j_0})$.
Since we have showed that
${\rm Sing}_{\tilde{g}}(\mathcal{O}_K)\bigcap(\mathcal{O}_K^{\times})^2=\emptyset$,
then it infers that $S(\tilde{g}, (\mathcal{O}_K^{\times})^2)=\emptyset$.
By Lemma \ref{lem 2.2}, one has
\begin{equation*}
Z_{\tilde{g}}(s, \chi, (\mathcal{O}_K^{\times})^2)
=\dfrac{G_0(q^{-s})}{1-q^{-1-s}},
\end{equation*}
where $G_0(x)$ is a polynomial with complex coefficients depending on $\chi$.
Thus Lemma \ref{lem 2.8} is true in this case.

{\sc Case 2.} ${\rm ord}(\alpha_{i_0})\ne{\rm ord}(\beta_{j_0})$.
By symmetry, we can assume ${\rm ord}(\alpha_{i_0})>{\rm ord}(\beta_{j_0})$, i.e.
${\rm ord}(\beta_{j_0})=e_0$. Thus for any $(u, v)\in (\mathcal{O}_K^{\times})^2$,
we deduce that
$$|\tilde{g}(u, v)|=|\alpha_{i_0}u^{i_0}+\beta_{j_0}v^{j_0}|
=q^{-e_0}|\pi^{-e_0}\alpha_{i_0}u^{i_0}+\pi^{-e_0}\beta_{j_0}v^{j_0}|
=q^{-e_0}.$$
Then one has
\begin{align*}
Z_{\tilde{g}}(s, \chi, (\mathcal{O}_K^{\times})^2)
=&\int_{(\mathcal{O}_K^{\times})^2}\chi(ac\tilde{g}(u, v))|\tilde{g}(u, v)|^s|dudv|\\
=&q^{-e_0s}\int_{(\mathcal{O}_K^{\times})^2}\chi(ac\tilde{g}(u, v))|dudv|\\
:=&q^{-e_0s}c(\chi).
\end{align*}
In fact, Lemma \ref{lem 2.2} gives us that
$c(\chi)=v(\bar{f}, (\mathcal{O}_K^{\times})^2, \chi)$.
Then it easily follows from the definition of
$v(\bar{f}, (\mathcal{O}_K^{\times})^2, \chi)$ that
$|c(\chi)|_{\infty}\le 1$. Hence Lemma \ref{lem 2.8}
is proved in this case. 

This concludes the proof of Lemma \ref{lem 2.8}.
\end{proof}

\begin{lem}\label{lem 2.9} {\rm (\cite{[YH]}, Lemma 2.6)}
Let $\chi: \mathcal{O}_K^{\times}\rightarrow \mathbb{C}^{\times}$
be a multiplicative character of $\mathcal{O}_K^{\times}$ and
$m$ be any positive integer. Then
\begin{align*}
\int_{\mathcal{O}_K^{\times}}{\chi(ac (x^m))|dx|}
={\left\{\begin{array}{rl}
1-q^{-1},\ \ \ &{\rm if}\ \chi^m=\chi_{{\rm triv}},\\
0,\ \ \ &{\rm if}\ \chi^m\neq \chi_{{\rm triv}}.
\end{array}\right.}
\end{align*}
\end{lem}

\begin{lem}\label{lem 2.10}
Let $k, r$ and $j$ be positive integers such that $1\le j\le k+r$. Let
\begin{align*}
S_{k, r}(j):={\left\{\begin{array}{rl}
\sum\limits_{i=k-j}^k(-1)^{k-i}\binom{k+r}{i+r}\binom{i+r}{i+j-k},\ \
&1\le j\le k,\\
\sum\limits_{i=0}^k(-1)^{k-i}\binom{k+r}{i+r}\binom{i+r}{i+j-k},
\ \ &k+1\le j\le k+r.
\end{array}\right.}
\end{align*}
Then 
\begin{align*}
S_{k, r}(j)={\left\{\begin{array}{rl}
0,\ \ &1\le j\le k,\\
\binom{k+r}{j}\sum\limits_{i=0}^k(-1)^i\binom{j}{i},
\ \ &k+1\le j\le k+r.
\end{array}\right.}
\end{align*}
\end{lem}

\begin{proof}
We prove the lemma by considering the following two cases.

{\sc Case 1.} $1\le j\le k$. Then 
\begin{align}\label{2.8}
S_{k, r}(j)
=&\sum\limits_{i=k-j}^k(-1)^{k-i}\dbinom{k+r}{i+r}\dbinom{i+r}{i+j-k}\nonumber\\
=&\sum\limits_{i=k-j}^k(-1)^{k-i}\dbinom{k+r}{k-i}\dbinom{i+r}{j-(k-i)}\nonumber\\
=&\sum\limits_{i=0}^j(-1)^i\dbinom{k+r}{i}\dbinom{k-i+r}{j-i}.
\end{align}

On the other hand, we have
\begin{align}\label{2.9}
\dbinom{k+r}{i}\dbinom{k-i+r}{j-i}
=&\dfrac{(k+r)!}{i!(k-i+r)!}\cdot\dfrac{(k-i+r)!}{(j-i)!(k+r-j)!}\nonumber\\
=&\dfrac{(k+r)!}{(k+r-j)!}\cdot\dfrac{1}{i!(j-i)!}\nonumber\\
=&\dbinom{k+r}{j}\cdot\dfrac{j!}{i!(j-i)!}\nonumber\\
=&\dbinom{k+r}{j}\dbinom{j}{i}.
\end{align}
Putting (\ref{2.9}) into (\ref{2.8}), we have
$$S_{k, r}(j)=\dbinom{k+r}{j}\Big(\sum\limits_{i=0}^j(-1)^i\dbinom{j}{i}\Big)
=\dbinom{k+r}{j}(1-1)^j=0$$
as expected. So Lemma \ref{lem 2.10} is proved in this case.

{\sc Case 2.} $k+1\le j\le k+r$. It then follows that
\begin{align*}
S_{k, r}(j)
=&\sum\limits_{i=0}^k(-1)^{k-i}\dbinom{k+r}{k-i}\dbinom{i+r}{j-(k-i)}\\
=&\sum\limits_{i=0}^k(-1)^i\dbinom{k+r}{i}\dbinom{k-i+r}{j-i}.
\end{align*}
Then applying (\ref{2.9}), we get the desire result. 
Hence Lemma \ref{lem 2.10} holds in this case. 

This completes the proof of Lemma \ref{lem 2.10}.
\end{proof}

\begin{lem}\label{lem 2.11}{\rm(Lemma 2.9, \cite{[YH]})}
Let $x=\pi^mx_1\in \mathcal{O}_K$ with $m$ being a nonnegative integer.
Then each of the following is true.

{\rm (1).} If ${\rm ord}(x)\ge m$, then $x_1\in \mathcal{O}_K$.

{\rm (2).} If ${\rm ord}(x)=m$, then $x_1\in \mathcal{O}_K^{\times}.$
\end{lem}

\begin{lem}\label{lem 2.12}
Let $g(u, v)=\sum\limits_{i=i_0}^{\infty}\alpha_iu^i
+\sum\limits_{j=j_0}^{\infty}\beta_jv^j$ be a polynomial 
of form (\ref{eqno 1.2}). Then 
\begin{align*}
Z_g(s, \chi, D)={\left\{\begin{array}{rl}
M_1(q^{-s}),\ \ &{\rm ord}(\alpha_{i_0})<{\rm ord}(\beta_{j_0}),\\
\dfrac{M_2(q^{-s})}{1-q^{-1-s}},\ \ &{\it otherwise},
\end{array}\right.}
\end{align*}
where $D=\mathcal{O}_K^{\times}\times\mathcal{O}_K$ 
and $M_1(x)$ and $M_2(x)$ are polynomials with complex 
coefficients depending on $\chi$.
\end{lem}

\begin{proof} Let $e_1={\rm ord}(\alpha_{i_0})$
and $e_2={\rm ord}(\beta_{j_0})$. We divide the proof into two cases.

{\sc Case 1.} $e_1<e_2$.
Because for all $i>i_0$, we have ${\rm ord}(\alpha_i)>0$ and
${\rm ord}(\beta_j)\ge e_2>e_1\ge 0$ for all $j\ge j_0$.
Hence for all $(u, v)\in \mathcal{O}_K^{\times}\times\mathcal{O}_K$,
one has
\begin{align*}
|g(u, v)|=&|\alpha_{i_0}u^{i_0}+\sum_{i=i_0+1}^{\infty}\alpha_iu^i
+\sum\limits_{j=j_0}^{\infty}\beta_jv^j|\\
=&q^{-e_1}|\pi^{-e_1}\alpha_{i_0}u^{i_0}+\pi^{-e_1}
\big(\sum_{i=i_0+1}^{\infty}\alpha_iu^i
+\sum\limits_{j=j_0}^{\infty}\beta_jv^j\big)|\\
=&q^{-e_1}.
\end{align*}
Therefore
$$Z_g(s, \chi, D)=q^{-e_1s}\int_D\chi(acg(u, v))|dudv|
:=M_1(q^{-s})$$
as expected. So Lemma \ref{lem 2.12} is proved in this case.

{\sc Case 2.} $e_1\ge e_2$. Then
$$Z_g(s, \chi, D)=q^{-e_2s}Z_{g_1}(s, \chi, D),$$
where $g_1(u, v)=\sum\limits_{i=i_0}^{\infty}\alpha_i^{'}u^i
+\sum\limits_{j=j_0}^{\infty}\beta_j^{'}v^j$ with
$\alpha_i^{'}=\pi^{-e_2}\alpha_i$ and $\beta_j^{'}=\pi^{-e_2}\beta_j$.
Hence it follows that
$\bar{g}_1(u, v)=\bar{\alpha_{i_0}^{'}}u^{i_0}+\bar{\beta_{j_0}^{'}}v^{j_0}$
with $\bar{\beta_{j_0}^{'}}\ne \bar{0}$ and
$\bar{D}=\mathbb{F}_q^{\times}\times \mathbb{F}_q$.

{\sc Subcase 2.1.} $j_0=1$. Then
$$\dfrac{\partial \bar{g}_1}{\partial v}(P)
=\bar{\beta}_{j_0}\ne\bar{0}$$
which shows that 
${\rm Sing}_{\bar{g}_1}(\mathbb{F}_q)\bigcap \bar{D}=\emptyset$,
which implies that $S(g_1, D)=\emptyset$. By Lemma \ref{lem 2.2}, 
one has
$$Z_g(s, \chi, D)=q^{-e_2s}Z_{g_1}(s, \chi, D)
=\dfrac{M_{2, 1}(q^{-s})}{1-q^{-1-s}}.$$
where $M_{2, 1}(x)$ is a polynomial with complex 
coefficients depending on $\chi$.
Hence Lemma \ref{lem 2.12} is true in this case.

{\sc Subcase 2.2.} $j_0>1$ and $e_1=e_2$, then $\bar{\alpha_{i_0}^{'}}\ne\bar{0}$.
Suppose $P=(u, v)\in {\rm Sing}_{\bar{g}_1}(\mathbb{F}_q)$, then
$$\bar{g}_1(P)=\dfrac{\partial \bar{g}_1}{\partial u}(P)
=\dfrac{\partial \bar{g}_1}{\partial v}(P)=\bar{0}.$$

If $i_0=1$, then
$\dfrac{\partial \bar{g}_1}{\partial u}(P)=\bar{\alpha_{i_0}^{'}}\ne\bar{0}$.
If $i_0>1$ and $p\nmid i_0$, then
$\dfrac{\partial \bar{g}_1}{\partial u}(P)=i_0\bar{\alpha_{i_0}^{'}}u^{i_0-1}=\bar{0}$
gives us $u=\bar{0}$.

If $p\mid i_0$, then $p\nmid j_0$ since $\gcd(i_0, j_0)=1$. Thus
$\dfrac{\partial \bar{g}_1}{\partial v}(P)=j_0\bar{\beta_{j_0}^{'}}v^{j_0-1}=\bar{0}$
tells us $v=\bar{0}$. But $\bar{g}_1(P)=\bar{0}$. Then one has $u=\bar{0}$.

But $\bar{D}=\mathbb{F}_q^{\times}\times \mathbb{F}_q$.
Hence in this case, we always have
${\rm Sing}_{\bar{g}_1}(\mathbb{F}_q)\bigcap \bar{D}=\emptyset$,
which implies that $S(g_1, D)=\emptyset$. Then Lemma \ref{lem 2.2} gives us that
$$Z_g(s, \chi, D)=q^{-e_2s}Z_{g_1}(s, \chi, D)=\dfrac{M_{2,2}(q^{-s})}{1-q^{-1-s}},$$
where $M_{2, 2}(x)$ is a polynomial with complex coefficients depending on $\chi$.
This finishes the proof of Lemma \ref{lem 2.12} in this case.

{\sc Subcase 2.3.} $j_0>1$ and $e_1>e_2$, then
$\bar{g}_1(u, v)=\bar{\beta_{j_0}^{'}}v^{j_0}$. Since $j_0>1$,
so it is easy to deduce that
$${\rm Sing}_{\bar{g}_1}(\mathbb{F}_q)\bigcap \bar{D}_1
=\{(u, v)\in \mathbb{F}_q^2|u\in \mathbb{F}_q^{\times}, v=\bar{0} \}.$$
Therefore $S(g_1, D)=\{(u, v)\in R^2|v=0, u\ne 0\}$ and
$D_{S(g_1, D)}=\mathcal{O}_K^{\times}\times\pi\mathcal{O}_K$.
By Lemma \ref{lem 2.2}, we arrive at
\begin{align}\label{2.10}
Z_g(s, \chi, D)=&q^{-e_2s}Z_{g_1}(s, \chi, D)\nonumber\\
=&q^{-e_2s}\Big(v(\bar{g}_1, D, \chi)
+\sigma(\bar{g}_1, D, \chi)\dfrac{(1-q^{-1})q^{-s}}{1-q^{-1-s}}
+Z_{g_1}(s, D_{S(g_1, D)})\Big)\nonumber\\
=&\dfrac{M_{2, 3}(q^{-s})}{1-q^{-1-s}}+q^{-e_2s}Z_{g_1}(s, \chi, D_{S(g_1, D)}),
\end{align}
where $M_{2, 3}(x)$ is a polynomial with complex coefficients depending on $\chi$.

For $Z_{g_1}(s, \chi, D_{S(g, D)})$, we make the change of variables
of the form: $(u, v)\mapsto (u_1, \pi v_1)$, then it follows that
\begin{align}\label{2.11}
Z_{g_1}(s, \chi, D_{S(g_1, D)})=q^{-1}Z_{g_2}(s, \chi, D),
\end{align}
where $g_2(u, v):=\sum\limits_{i=i_0}^{\infty}\alpha_i^{'}u^i
+\sum\limits_{j=j_0}^{\infty}\beta_j^{'}\pi^jv^j$.
Then putting (\ref{2.11}) into (\ref{2.10}) yields that
\begin{equation*}
Z_g(s, \chi, D)=\dfrac{M_{2, 3}(q^{-s})}{1-q^{-1-s}}
+q^{-1-e_2s}Z_{g_2}(s, \chi, D).
\end{equation*}

Let $e_1-e_2=dj_0+e_3$ with $d$ and $e_3$ being positive integers and
$0\le e_3<j_0$. By Lemma \ref{lem 2.2} applied to the polynomial
$g_2$ for $d$ times, the above argument finally yields that
\begin{equation}\label{eqno 2.12}
Z_g(s, \chi, D)=\dfrac{M_{2, 4}(q^{-s})}{1-q^{-1-s}}
+q^{-d-1-dj_0s-e_2s}Z_{g_3}(s, \chi, D),
\end{equation}
where $g_3(u, v):=\sum\limits_{i=i_0}^{\infty}\alpha_i^{'}\pi^{-dj_0}u^i
+\sum\limits_{j=j_0}^{\infty}\beta_j^{'}\pi^{(d+1)j-dj_0}v^j$
and $M_{2, 4}(x)$ is a polynomial with complex coefficients depending on $\chi$.

Since ${\rm ord}(\alpha_{i_0}^{'}\pi^{-dj_0})=e_3$,
and ${\rm ord}(\beta_{j_0}^{'}\pi^{(d+1)j_0-dj_0})=j_0$,
so the result for {\sc Case 1} tells us that
\begin{equation}\label{eqno 2.13}
Z_{g_3}(s, \chi, D)=M_{2, 5}(q^{-s}),
\end{equation}
where $M_{2, 5}(x)$ is a polynomial with complex coefficients depending on $\chi$.

Now from (\ref{eqno 2.12}) and (\ref{eqno 2.13}), we derive that
\begin{equation*}
Z_g(s, \chi, D)=\dfrac{M_2(q^{-s})}{1-q^{-1-s}}
\end{equation*}
as expected. So \ref{lem 2.12} is true for this Case.

This concludes the proof of Lemma \ref{lem 2.12} for this Case.
\end{proof}

Finally, using Theorem \ref{thm 1.1}, we give the following lemma 
as the conclusion of this section, which will be used in the proof 
of Theorem \ref{thm 1.2}.

\begin{lem}\label{lem 2.13}
Let $D=(\mathcal{O}_K^{\times})^2$ and
$$h(u, v)=u^p+v^r\mathbb{H}_r^k(v, t-v)
=u^p+v^r\sum\limits_{i=0}^k\binom{k+r}{i+r}v^i(t-v)^{k-i},$$ 
where $r$ and $k$ are positive integers such that $p\nmid r(k+r)$, 
$p\nmid \binom{k+r}{r}$ and $t^{1/p}\in \mathcal{O}_K^{\times}$.
Then
\begin{align*}
Z_h(s, \chi, D)=\dfrac{L(q^{-s})}{(1-q^{-1-s})(1-q^{-p-(k+1)-p(k+1)s})},
\end{align*}
where $L(x)$ is a polynomial with complex coefficients depending on $\chi$.
\end{lem}

\begin{proof}
First of all, we have
$$\bar{h}(u, v)=u^p+v^r\sum\limits_{i=0}^{k}
{\overline{\dbinom{k+r}{i+r}}v^i(\bar{t}-v)^{k-i}},$$
and $\bar{D}=(\mathbb{F}_q^{\times})^2$.
Then using Lemma \ref{lem 2.1}, we obtain that
\begin{equation*}
\dfrac{\partial \bar{h}}{\partial v}(u, v)=\dfrac{\partial}{\partial v}
(v^r\mathbb{H}_r^k(v, t-v))=\bar{r}\overline{\dbinom{k+r}{r}}v^{r-1}(\bar{t}-v)^k.
\end{equation*}
Notice that for any $n\in \mathbb{Z}$
with $p\nmid n$, there exists an integer $n^{'}$ such that $nn^{'}\equiv 1\pmod{p}$.
So $nn^{'}=1$ since ${\rm char}(K)=p$, which implies that $n\in\mathcal{O}_K^{\times}$.
Then it follows that $\bar{r}\overline{\binom{k+r}{r}}\ne \bar{0}$
since $p\nmid r\binom{k+r}{r}$.

If $P=(u, v)\in {\rm Sing}_{\bar{h}}(\mathbb{F}_q)$, then
$\frac{\partial \bar{h}}{\partial v}(P)$=0 tells us $v=\bar{t}$
since $v\in \mathbb{F}_q^{\times}$. Hence $\bar{h}(P)=0$ gives us that
$$u^p+\bar{t}^{k+r}=0.$$
Because $t^{(k+r)/p}\in \mathcal{O}_K^{\times}$, so we have $u=-\bar{t}^{(k+r)/p}$.
Then it follows that ${\rm Sing}_{\bar{h}}(\mathbb{F}_q)\cap \bar{D}=(-\bar{t}^{(k+r)/p}, \bar{t})$,
which implies that $S(h, D)=(-t^{(k+r)/p}, t)$ and
$$D_{S(h, D)}=(-t^{(k+r)/p}+\pi\mathcal{O}_K)\times(t+\pi\mathcal{O}_K).$$
Therefore by Lemma \ref{lem 2.2}, we obtain that
\begin{align}\label{2.14}
Z_h(s, \chi, D)=&v(\bar{h}, D, \chi)+\sigma(\bar{h}, D, \chi)
\dfrac{(1-q^{-1})q^{-s}}{1-q^{-1-s}}+Z_h(s, \chi, D_{S(h, D)})\nonumber\\
:=&\dfrac{H_1(q^{-s})}{1-q^{-1-s}}+Z_h(s, \chi, D_{S(h, D)})
\end{align}
where $H_1(x)$ is a polynomial with complex coefficients depending on $\chi$.

For $Z_h(s, \chi, D_{S(h, D)})$, we make the change of variables of form:
$(u, v)\mapsto (-t^{(k+r)/p}+\pi u_1, t+\pi v_1)$, and arrive at
$$Z_h(s, \chi, D_{S(h, D)})=q^{-2}Z_{h_1}(s, \chi),$$
where
\begin{align*}
h_1(u, v)=&(-t^{(k+r)/p}+\pi u)^p+\sum\limits_{i=0}^k\dbinom{k+r}{i+r}
(t+\pi v)^{r+i}(-\pi v)^{k-i}\\
=&\pi^pu^p+(-t^{(k+r)/p})^p+t^{k+r}+\sum_{j=1}^{k+r}S_{k, r}(j)t^{k+r-j}\pi^jv^j\\
=&\pi^pu^p+\sum_{j=1}^{k+r}S_{k, r}(j)t^{k+r-j}\pi^jv^j,
\end{align*}
where $S_{k, r}(j)$ are integers defined as in Lemma \ref{lem 2.10}.
Therefore, Lemma \ref{lem 2.10} yields that
\begin{equation*}
h_1(u, v)=\pi^pu^p+\sum_{j=k+1}^{k+r}S_{k, r}(j)t^{k+r-j}\pi^jv^j.
\end{equation*}
Furthermore, we derive that
\begin{align*}
S_{k, r}(k+1)=&\dbinom{k+r}{k+1}\Big(\sum\limits_{i=0}^k(-1)^i\dbinom{k+1}{i}\Big)\\
=&\dbinom{k+r}{k+1}\Big(\sum\limits_{i=0}^{k+1}(-1)^i\dbinom{k+1}{i}-(-1)^{k+1}\Big)\\
=&(-1)^k\dbinom{k+r}{k+1}.
\end{align*}
Oh the other hand, one has
\begin{align*}
\dbinom{k+r}{k+1}=\dbinom{k+r}{r-1}=\dfrac{r}{k+1}\dbinom{k+r}{r}.
\end{align*}
Since $p\nmid r\binom{k+r}{r}$, we have $p\nmid (k+1)\binom{k+r}{k+1}$,
which implies that $\gcd(p, k+1)=1$ and ${\rm ord}((-1)^k\binom{k+r}{k+1}t^{r-1})=1$.
Hence $h_1$ is a polynomial of form (\ref{eqno 1.2}). By Theorem \ref{thm 1.1},
we obtain that
\begin{align}\label{2.15}
Z_h(s, \chi, D_{S(h, D)})=q^{-2}Z_{h_1}(s, \chi)
=\dfrac{H_2(q^{-s})}{(1-q^{-1-s})(1-q^{-p-(k+1)-p(k+1)s})},
\end{align}
where $H_2(x)$ is a polynomial with complex coefficients depending on $\chi$.
From (\ref{2.15}) and (\ref{2.14}), we get that
$$Z_h(s, \chi, D)=\dfrac{L(q^{-s})}{(1-q^{-1-s})(1-q^{-p-(k+1)-p(k+1)s})},$$
where $L(x)$ is a polynomial with complex coefficients depending on $\chi$.

Thus the proof of Lemma \ref{lem 2.13} is complete.
\end{proof}

\section{\bf Proof of Theorem \ref{thm 1.1}}

In this section, we present the proof of Theorem \ref{thm 1.1}.
First of all, for any positive integer $n$ and given nonnegative
integers $t_1$, $\cdots$, $t_n$, we introduce the transform
$T_{t_1, \cdots, t_n}$ of variables defined as follows:
\begin{equation}\label{eqno 3.1}
T_{t_1, \cdots, t_n}(x_1, \cdots, x_n):=(\pi^{t_1}y_1, \cdots, \pi^{t_n}y_n).
\end{equation}
Let $J(T_{t_1, \cdots, t_n})$ be the Jacobian determinant associated to
$T_{t_1, \cdots, t_n}$, then by (\ref{eqno 3.1}),
we deduce that $J(T_{t_1, \cdots, t_n})=\pi^{t_1+\cdots+t_n}$.
It follows that for any $f(x)\in \mathcal{O}_K[x]$ and $D\subseteq \mathcal{O}_K^n$,
we have
\begin{align}\label{3.2}
&Z_f(s, \chi, D)\nonumber\\
=&\int_{D}\chi(acf(x))|f(x)|^s|dx|\nonumber\\
=&\int_{D_1}\chi(acf(\pi^{t_1}y_1, \cdots, \pi^{t_n}y_n))
|f(\pi^{t_1}y_1, \cdots, \pi^{t_n}y_n)|^s
|J(T_{t_1, \cdots, t_n})||dy_1\cdots dy_n|\nonumber\\
=&q^{-(t_1+\cdots+t_n)}\int_{D_1}\chi(acf(\pi^{t_1}y_1, \cdots, \pi^{t_n}y_n))
|f(\pi^{t_1}y_1, \cdots, \pi^{t_n}y_n)|^s|dy|,
\end{align}
where $D_1$ is the domain of $(y_1, \cdots, y_n)$.
The transform (\ref{eqno 3.1}) will be repeatedly used in the proofs of
Theorems \ref{thm 1.1} and \ref{thm 1.2}.

Moreover, let $C$ be any set with $C\subseteq\mathbb{N}^n$,
we define a set associated to $C$:
$$S(C):=\{(x_1,\cdots, x_n)\in \mathcal{O}_K^{n}|
\big({\rm ord}(x_1),\cdots, {\rm ord}(x_n)\big)\in C\}.$$
Then by (\ref{eqno 2.4}), one has
\begin{equation}\label{eqno 3.3}
Z_f(s, \chi)=Z_f\big(s, \chi, (\mathcal{O}_K^{\times})^n\big)
+\sum_{\gamma}Z_f\big(s, \chi, S(\Delta_{\gamma}\bigcap (\mathbb{N}^n\setminus \{0\}))\big),
\end{equation}
where $\gamma$ runs over all the proper faces of $\Gamma(f)$.

Now we can give the proof of Theorem \ref{thm 1.1}.\\

{\it Proof of Theorem \ref{thm 1.1}.}
Let $g(u, v)=\sum\limits_{i=i_0}^{\infty}\alpha_iu^i
+\sum\limits_{j=j_0}^{\infty}\beta_jv^j$ be a polynomial of form (\ref{eqno 1.2}),
$e_1={\rm ord}(\alpha_{i_0})$ and $e_2={\rm ord}(\beta_{j_0})$.
By (\ref{eqno 3.3}), we have
\begin{align}\label{3.4}
Z_g(s, \chi)=Z_g\big(s, \chi, (\mathcal{O}_K^{\times})^2)
+\sum_{i=1}^{5}Z_g\Big(s, \chi, S(\Delta_{\gamma_i}\bigcap(\mathbb{N}^2\setminus \{0\}))\Big),
\end{align}
where $\gamma_i$ is defined in the proof of Lemma \ref{lem 2.4}.
To prove Theorem \ref{thm 1.1}, we need to deal with the six integrals
on the right hand side of (\ref{3.4}) respectively.

For $Z_g(s, \chi, (\mathcal{O}_K^{\times})^2)$, Lemma \ref{lem 2.8} gives us that
\begin{equation}\label{eqno 3.5}
Z_g(s, \chi, (\mathcal{O}_K^{\times})^2)=\dfrac{G_0(q^{-s})}{1-q^{-1-s}},
\end{equation}
where $G_0(x)$ is a polynomial with complex coefficients depending on $\chi$.

For $Z_g\big(s, \chi, S(\Delta_{\gamma_1}\bigcap(\mathbb{N}^2\setminus \{0\})\big)$,
Lemma \ref{lem 2.5} tells us that
$\Delta_{\gamma_1}\bigcap(\mathbb{N}^2\setminus \{0\})=\{(0, a)|a\in \mathbb{Z}^{+}\}$.
Thus it is easy to see that
\begin{equation*}
S(\Delta_{\gamma_1}\bigcap(\mathbb{N}^2\setminus \{0\}))
=\bigcup_{a=1}^{\infty}(\mathcal{O}_K^{\times}\times\pi^a\mathcal{O}_K^{\times}).
\end{equation*}
Then we deduce that
$$Z_g\big(s, \chi, S(\Delta_{\gamma_1}\bigcap(\mathbb{N}^2\setminus \{0\}))\big)
=\sum_{a=1}^{\infty}Z_g(s, \chi, D_1(a)),$$
where $D_1(a)=\mathcal{O}_K^{\times}\times\pi^a\mathcal{O}_K^{\times}$.
For $Z_g(s, \chi, D_1(a))$, we set $t_1=0$ and $t_2=a$ in (\ref{eqno 3.1}).
Then by (\ref{3.2}) and Lemma \ref{lem 2.11}, one derives that
\begin{align}\label{3.6}
Z_g\big(s, \chi, S(\Delta_{\gamma_1}\bigcap(\mathbb{N}^2\setminus \{0\}))\big)
=\sum_{a=1}^{\infty}q^{-a}Z_{g_{1, a}}(s, (\mathcal{O}_K^{\times})^2),
\end{align}
where $g_{1, a}(u, v)=\sum\limits_{i=i_0}^{\infty}\alpha_iu^i
+\sum\limits_{j=j_0}^{\infty}\beta_j\pi^{aj}v^j$.
Notice that for all $j>j_0$, one has
\begin{align*}
{\rm ord}(\beta_j\pi^{aj})=&{\rm ord}(\beta_j)+{\rm ord}(\pi^{aj})
>{\rm ord}(\beta_{j_0})+{\rm ord}(\pi^{aj_0})
={\rm ord}(\beta_{j_0}\pi^{aj_0}),
\end{align*}
and ${\rm ord}(\alpha_i)>{\rm ord}(\alpha_{i_0})$ for all
$i>i_0$. Hence $g_{1, a}$ are polynomials of form (\ref{eqno 1.2}),
then Lemma \ref{lem 2.8} gives us that
$$Z_{g_{1, a}}(s, \chi, (\mathcal{O}_K^{\times})^2)=
Z_{\tilde{g}_{1, a}}(s, \chi, (\mathcal{O}_K^{\times})^2),$$
where $\tilde{g}_{1, a}(u, v)=\alpha_{i_0}u^{i_0}+\beta_{j_0}\pi^{aj_0}v^{j_0}$.

Let $d_1=\lfloor\frac{e_1-e_2}{j_0}\rfloor$ if $e_1\ge e_2+j_0$, or $d_1=0$.
When $1\le a\le d_1$(if $a$ exists), by Lemma \ref{lem 2.8}, one has
\begin{equation}\label{eqno 3.7}
Z_{\tilde{g}_{1, a}}(s, \chi, (\mathcal{O}_K^{\times})^2)=
\dfrac{G_{1, a}(q^{-s})}{1-q^{-1-s}},
\end{equation}
where $G_{1, a}(x)$ is a polynomial with complex coefficients 
depending on $\chi$ and $a$.

When $a\ge d_1+1$, one deduces that
$${\rm ord}(\beta_{j_0}\pi^{aj_0})=e_2+aj_0>e_1={\rm ord}(\alpha_{i_0}).$$
Hence by Lemma \ref{lem 2.8}, it follows that
$$Z_{\tilde{g}_{1, a}}(s, \chi, (\mathcal{O}_K^{\times})^2)=c_1(\chi, a)q^{-e_1s},$$
where $c_1(\chi, a)$ is a constant depend on $\chi$ and $a$
with $|c_1(\chi, a)|_{\infty}\le 1$. Since the series
$\sum\limits_{a=d_1+1}^{\infty}q^{-a}c_1(\chi, a)$ converges, we derive that
\begin{align}\label{3.8}
\sum_{a=d_1+1}^{\infty}q^{-a}Z_{\tilde{g}_{1, a}}(s, \chi, (\mathcal{O}_K^{\times})^2)
=&q^{-e_1s}\sum_{a=d_1+1}^{\infty}q^{-a}c_1(\chi, a)\nonumber\\
:=&c_1(\chi)q^{-e_1s},
\end{align}
where $c_1(\chi)$ is a constant depend on $\chi$.

Putting (\ref{eqno 3.7}) and (\ref{3.8}) into (\ref{3.6}),
we get that
\begin{equation}\label{eqno 3.9}
Z_g\big(s, \chi, S(\Delta_{\gamma_1}\bigcap(\mathbb{N}^2\setminus \{0\}))\big)
=\dfrac{G_1(q^{-s})}{1-q^{-1-s}},
\end{equation}
where $G_1(x)$ is a polynomial with complex coefficients depending on $\chi$.

For $Z_g\big(s, \chi, S(\Delta_{\gamma_2}\bigcap(\mathbb{N}^2\setminus \{0\}))\big)$,
By Lemma \ref{lem 2.5}, we have
\begin{align*}
\Delta_{\gamma_2}\bigcap (\mathbb{N}^2\setminus \{0\})
=\bigcup\limits_{m=0}^{j_0-1}\big\{(m+bj_0, \frac{mi_0+n_m}{j_0}+bi_0+a)
|a, b\in \mathbb{Z}^{+}\}
\end{align*}
Denote $w(m):=\frac{mi_0+n_m}{j_0}$. Then
\begin{equation*}
S(\Delta_{\gamma_2}\bigcap(\mathbb{N}^2\setminus \{0\}))=
\bigcup_{m=0}^{j_0-1}\bigcup_{a=1}^{\infty}\bigcup_{b=1}^{\infty}
(\pi^{m+bj_0}\mathcal{O}_K^{\times}
\times\pi^{w(m)+a+bi_0}\mathcal{O}_K^{\times}).
\end{equation*}
Thus one has
\begin{equation}\label{eqno 3.10}
Z_g\big(s, \chi, S(\Delta_{\gamma_2}\bigcap(\mathbb{N}^2\setminus \{0\}))\big)
=\sum_{m=0}^{j_0-1}\sum_{a=1}^{\infty}\sum_{b=1}^{\infty}
Z_{g}(s, \chi, D_2(a, b)),
\end{equation}
where $D_2(a, b)=\pi^{m+bj_0}\mathcal{O}_K^{\times}
\times\pi^{w(m)+a+bi_0}\mathcal{O}_K^{\times}$.
For $Z_{g}(s, \chi, D_2(a, b))$, we set $t_1=m+bj_0$ 
and $t_2=w(m)+a+bi_0$ in (\ref{eqno 3.1}).
Then by Lemma \ref{lem 2.11} and (\ref{3.2}), one has
\begin{align}\label{3.11}
&Z_g\big(s, \chi, S(\Delta_{\gamma_2}\bigcap(\mathbb{N}^2\setminus \{0\}))\big)\nonumber\\
=&\sum_{m=0}^{j_0-1}\sum_{a=1}^{\infty}
\sum_{b=1}^{\infty}q^{-(m+w(m)+a+b(i_0+j_0))}\nonumber\\
\times&\Big(\int_{(\mathcal{O}_K^{\times})^2}
\chi(ac(\sum_{i=i_0}^{\infty}\alpha_i\pi^{(m+bj_0)i}u^i
+\sum_{j=j_0}^{\infty}\beta_j\pi^{(w(m)+a+bi_0)j}v^j))\nonumber\\
\times&|\sum_{i=i_0}^{\infty}\alpha_i\pi^{(m+bj_0)i}u^i
+\sum_{j=j_0}^{\infty}\beta_j\pi^{(w(m)+a+bi_0)j}v^j|^s|dudv|\Big)\nonumber\\
:=&\sum_{b=1}^{\infty}q^{-b(i_0+j_0+i_0j_0s)}\sum_{m=0}^{j_0-1}q^{-m-w(m)}
\sum_{a=1}^{\infty}q^{-a}Z_{g_{2, a}}(s, \chi, (\mathcal{O}_K^{\times})^2),
\end{align}
where
$$g_{2, a}(u, v)=\sum_{i=i_0}^{\infty}\alpha_i\pi^{(m+bj_0)i-bi_0j_0}u^i
+\sum_{j=j_0}^{\infty}\beta_j\pi^{(w(m)+a+bi_0)j-bi_0j_0}v^j,$$

Since for all $i>i_0$ and $j>j_0$, we have
${\rm ord}(\alpha_i)>e_1$ and ${\rm ord}(\beta_j)>e_2$.
It is easy to see that
$${\rm ord}(\alpha_i\pi^{(m+bj_0)i-bi_0j_0})
>{\rm ord}(\alpha_{i_0}\pi^{(m+bj_0)i_0-bi_0j_0})=e_1+mi_0,$$
and
$${\rm ord}(\beta_j\pi^{(w(m)+a+bi_0)j-bi_0j_0})
>{\rm ord}(\beta_{j_0}\pi^{(w(m)+a+bi_0)j_0-bi_0j_0})=e_2+w(m)+aj_0.$$
Hence $g_{2, a}$ are polynomials of form (\ref{eqno 1.2}). 
Then Lemma \ref{lem 2.8} gives us that
$$Z_{g_{2, a}}(s, \chi, (\mathcal{O}_K^{\times})^2)=
Z_{\tilde{g}_{2, a}}(s, \chi, (\mathcal{O}_K^{\times})^2),$$
where $\tilde{g}_{2, a}(u, v)=\alpha_{i_0}\pi^{mi_0}u^{i_0}
+\beta_{j_0}\pi^{(w(m)+a)j_0}v^{j_0}$ (notice that $\tilde{g}_{2, a}$
are independent of $b$).

If $1\le a\le d_2$ (if $a$ exists), using Lemma \ref{lem 2.8}, one has
\begin{equation}\label{eqno 3.12}
Z_{\tilde{g}_{2, a}}(s, \chi, (\mathcal{O}_K^{\times})^2)
=\dfrac{G_{2, a, m}(q^{-s})}{1-q^{-1-s}},
\end{equation}
where $G_{2, a, m}(x)$ is a polynomial with complex coefficients depending on
$\chi$, $a$ and $m$, but independent of $b$.

If $a\ge d_2+1$, then $e_2+w(m)+aj_0>e_1+mi_0$, that is
$${\rm ord}(\beta_{j_0}\pi^{(w(m)+a+bi_0)j_0-bi_0j_0})>
{\rm ord}(\alpha_{i_0}\pi^{(m+bj_0)i_0-bi_0j_0}).$$

Thus by Lemma \ref{lem 2.8}, it follows that
\begin{equation}\label{eqno 3.13}
Z_{\tilde{g}_{2, a}}(s, \chi, (\mathcal{O}_K^{\times})^2)
=c_2(\chi, a, m)q^{-(e_1+mi_0)s},
\end{equation}
where $c_2(\chi, a, m)$ is a constant depend on $\chi$, $a$ and $m$
with $|c_2(\chi, a, m)|_{\infty}\le 1$. Since the series
$\sum\limits_{a=d_2+1}^{\infty}q^{-a}c_2(\chi, a, m)$ converges,
then putting (\ref{eqno 3.12}) and (\ref{eqno 3.13}) into (\ref{3.11}),
we obtain that
\begin{align}\label{3.14}
Z_g\big(s, \chi, S(\Delta_{\gamma_2}\bigcap(\mathbb{N}^2\setminus \{0\}))\big)
=&\sum_{m=0}^{j_0-1}q^{-m-w(m)}\dfrac{\tilde{G}_2(q^{-s})}{1-q^{-1-s}}
\sum_{b=1}^{\infty}q^{-b(i_0+j_0+i_0j_0s)}\nonumber\\
=&\dfrac{G_2(q^{-s})}{(1-q^{-1-s})(1-q^{-i_0-j_0-i_0j_0s})},
\end{align}
where $G_2(x)$ is a polynomial with complex coefficients depending on $\chi$.

For $Z_g\big(s, \chi, S(\Delta_{\gamma_3}\bigcap(\mathbb{N}^2\setminus \{0\}))\big)$,
Lemma \ref{lem 2.5} tells us that
$$\Delta_{\gamma_3}\bigcap (\mathbb{N}^2\setminus \{0\})
=\{(aj_0, ai_0)|a\in \mathbb{Z}^{+}\}.$$
Then one derives that
\begin{equation*}
S(\Delta_{\gamma_3}\bigcap (\mathbb{N}^2\setminus \{0\}))
=\bigcup_{a=1}^{\infty}(\pi^{aj_0}\mathcal{O}_K^{\times}
\times\pi^{ai_0}\mathcal{O}_K^{\times}).
\end{equation*}
So
$$Z_g\big(s, \chi, S(\Delta_{\gamma_1}\bigcap(\mathbb{N}^2\setminus \{0\}))\big)
=\sum_{a=1}^{\infty}Z_g(s, \chi, D_3(a)),$$
where $D_3(a)=\pi^{aj_0}\mathcal{O}_K^{\times}\times\pi^{ai_0}\mathcal{O}_K^{\times}$.
For $Z_g(s, \chi, D_3(a))$, we set $t_1=aj_0$ and $t_2=ai_0$ in (\ref{eqno 3.1}).
Then from Lemma \ref{lem 2.10} and (\ref{3.2}), we get that
\begin{align}\label{3.15}
&Z_g\big(s, \chi, S(\Delta_{\gamma_1}\bigcap(\mathbb{N}^2\setminus \{0\}))\big)\nonumber\\
=&\sum_{a=1}^{\infty}q^{-a(i_0+j_0)}\nonumber\\
\times&\int_{(\mathcal{O}_K^{\times})^2}
\chi(ac(\sum_{i=i_0}^{\infty}\alpha_i\pi^{aj_0i}u^i
+\sum_{j=j_0}^{\infty}\beta_j\pi^{ai_0j}v^j))
|\sum_{i=i_0}^{\infty}\alpha_i\pi^{aj_0i}u^i
+\sum_{j=j_0}^{\infty}\beta_j\pi^{ai_0j}v^j|^s|dudv|\nonumber\\
:=&\sum_{a=1}^{\infty}q^{-a(i_0+j_0+i_0j_0s)}
Z_{g_{3, a}}(s, \chi, (\mathcal{O}_K^{\times})^2),
\end{align}
where
$$g_{3, a}(u, v)=\sum_{i=i_0}^{\infty}\alpha_i\pi^{aj_0(i-i_0)}u^i
+\sum_{j=j_0}^{\infty}\beta_j\pi^{ai_0(j-j_0)}v^j.$$
It is no hard to see that $g_{3, a}$ are polynomials of form (\ref{eqno 1.2}),
so Lemma \ref{lem 2.8} gives us that
$$Z_{g_{3, a}}(s, \chi, (\mathcal{O}_K^{\times})^2)=
Z_{g_3}(s, \chi, (\mathcal{O}_K^{\times})^2),$$
where $g_3(u, v):=\alpha_{i_0}u^{i_0}+\beta_{j_0}v^{j_0}$.
By Lemma \ref{lem 2.2}, one has
\begin{equation}\label{eqno 3.16}
Z_{g_3}(s, \chi, (\mathcal{O}_K^{\times})^2)=\dfrac{\tilde{G}_3(q^{-s})}{1-q^{-1-s}},
\end{equation}
where $\tilde{G}_3(x)$ is a polynomial with complex coefficients
depending on $\chi$ but independent on $a$.

Thus using (\ref{3.15}) and (\ref{eqno 3.16}), we arrive at
\begin{align}\label{3.17}
Z_g\big(s, \chi, S(\Delta_{\gamma_3}\bigcap(\mathbb{N}^2\setminus \{0\}))\big)
=&\dfrac{\tilde{G}_3(q^{-s})}{1-q^{-1-s}}\sum_{a=1}^{\infty}q^{-a(i_0+j_0+i_0j_0s)}\nonumber\\
=&\dfrac{G_3(q^{-s})}{(1-q^{-1-s})(1-q^{-i_0-j_0-i_0j_0s})},
\end{align}
where $G_3(x)$ is a polynomial with complex coefficients depending on $\chi$.

For
$Z_g\big(s, \chi, S(\Delta_{\gamma_4}\bigcap(\mathbb{N}^2\setminus \{0\}))\big)$,
it is easy to see that this case is symmetry to the
case $i=2$. Thus use the same discussion, we have
\begin{equation}\label{eqno 3.18}
Z_g\big(s, \chi, S(\Delta_{\gamma_4}\bigcap(\mathbb{N}^2\setminus \{0\}))\big)
=\dfrac{G_4(q^{-s})}{(1-q^{-1-s})(1-q^{-i_0-j_0-i_0j_0s})},
\end{equation}
where $G_4(x)$ is a polynomial with complex coefficients depending on $\chi$.

Similarly, the case $i=5$ is symmetry to the case $i=1$.
Hence it follows that
\begin{equation}\label{eqno 3.19}
Z_g\big(s, \chi, S(\Delta_{\gamma_5}\bigcap(\mathbb{N}^2\setminus \{0\}))\big)
=\dfrac{G_5(q^{-s})}{1-q^{-1-s}},
\end{equation}
where $G_5(x)$ is a polynomial with complex coefficients depending on $\chi$.

Put (\ref{eqno 3.5}), (\ref{eqno 3.9}), (\ref{3.14}), (\ref{3.17}),
(\ref{eqno 3.18}) and (\ref{eqno 3.19}) into (\ref{3.4}), we derive that
$$Z_g(s, \chi)=\dfrac{G(q^{-s})}{(1-q^{-1-s})(1-q^{-i_0-j_0-i_0j_0s})},$$
where $G(q^{-s})$ is a polynomial with complex coefficients depending on $\chi$.

This finishes the proof of Theorem \ref{thm 1.1}.
\hfill$\Box$

\section{\bf Proof of Theorem \ref{thm 1.2}}

In this section, we give the proof of Theorem \ref{thm 1.2}.
In fact, we can use Theorem \ref{thm 1.1} to show Theorem \ref{thm 1.2}.\\

{\it Proof of Theorem \ref{thm 1.2}}.
Let $f(x, y, z)=x^p+y^rz^l\mathbb{H}_r^k(y, tz-y)$
be a hybrid polynomial of form (\ref{1.1}) with 
$t^{1/p}\in \mathcal{O}_K^{\times}$.
We divide the proof into the following two cases.

{\sc Case 1.} $p\mid\binom{k+r}{r}$. We claim that 
there is a polynomial $h(y, z)\in\mathcal{O}_K[y, z]$ such that 
\begin{equation*}
f(x, y, z)=x^p+h(y, z)^p.
\end{equation*}
In fact, by (\ref{eqno 2.1}), we have
\begin{align*}
f(x, y, z)=&x^p+\sum_{i=0}^k(-1)^i\dbinom{k+r}{k-i}
\dbinom{i+r-1}{i}t^{k-i}y^{r+i}z^{k+l-i}\\
=&x^p+\sum_{i=0}^k(-1)^ia_{k, r}(i)t^{k-i}y^{r+i}z^{k+l-i},
\end{align*}
where $a_{k, r}(i):=\binom{k+r}{k-i}\binom{i+r-1}{i}$.
On the other hand, for any $0\le i\le k$, one has
\begin{align*}
a_{k, r}(i)=&\dfrac{(k+r)!}{(r+i)!(k-i)!}\cdot
\dfrac{(i+r-1)!}{(r-1)!i!}\\
=&\dfrac{r}{r+i}\cdot\dfrac{k!}{(k-i)!i!}\cdot
\dfrac{(k+r)!}{r!k!}\\
=&\dfrac{r}{r+i}\dbinom{k}{i}\dbinom{k+r}{r}.
\end{align*}

If $p\mid a_{k, r}(i)$, we must have $a_{k, r}(i)=0$ 
over $K$ since the characteristic of $K$ is $p$.

If $p\nmid a_{k, r}(i)$, then we have $p\mid (r+i)$ since
$p\mid\binom{k+r}{r}$. But $p\mid(k+r+l)$, and so 
$p\mid(k+l-i)$. Moreover, one has $a_{k, r}(i)^p=a_{k, r}(i)$ 
over $K$. Since $t^{1/p}\in \mathcal{O}_K^{\times}$
and $(-1)^{p}=-1\in \mathcal{O}_K^{\times}$, it follows that
\begin{align*}
f(x, y, z)=&x^p+\sum_{i=0}^k(-1)^ia_{k, r}(i)t^{k-i}y^{r+i}z^{k+l-i}\\
=&x^p+\sum_{i=0 \atop p\nmid a_{k, r}(i)}^k
\Big((-1)^ia_{k, r}(i)t^{(k-i)/p}y^{(r+i)/p}z^{(k+l-i)/p}\Big)^p\\
=&x^p+\Big(\sum_{i=0 \atop p\nmid a_{k, r}(i)}^k
(-1)^ia_{k, r}(i)t^{(k-i)/p}y^{(r+i)/p}z^{(k+l-i)/p}\Big)^p\\
:=&x^p+h(y, z)^p
\end{align*}
as desired. So the claim is true.

Now for $Z_f(s, \chi)$, we make the following change of variables:
$$(x, y, z)\mapsto (x-h(y, z), y, z).$$ 
Since $|dx|$ is the Haar measure on $K$, and normalized 
such that the measure of $\mathcal{O}_K$ is one, we have 
\begin{align}\label{4.1}
Z_f(s, \chi)=&\int_{\mathcal{O}_K^3}
\chi(ac((x-h(y, z))^p+h(y, z)^p))|(x-h(y, z))^p+h(y, z)^p|^s|dxdydz|\nonumber\\
=&\int_{\mathcal{O}_K}\chi(ac(x^p))|x^p|^s|dx|\int_{\mathcal{O}_K^2}|dydz|\nonumber\\
=&\int_{\mathcal{O}_K}\chi(ac(x^p))|x^p|^s|dx|\nonumber\\
=&\int_{\mathcal{O}_K^{\times}}\chi(ac(x^p))|dx|
+\int_{\pi\mathcal{O}_K}\chi(ac(x^p))|x^p|^s|dx|\nonumber\\
=&\int_{\mathcal{O}_K^{\times}}\chi(ac(x^p))|dx|+q^{-1-ps}Z_f(s, \chi).
\end{align}
By (\ref{4.1}), we arrive at
$$(1-q^{-1-ps})Z_f(s, \chi)=\int_{\mathcal{O}_K^{\times}}\chi(ac(x^p))|dx|.$$
Therefore Lemma \ref{lem 2.9} gives us that
\begin{align*}
Z_f(s, \chi)={\left\{\begin{array}{rl}
\dfrac{1-q^{-1}}{1-q^{-1-ps}}, \ \ \ \ &{\rm if} \ \chi^p=\chi_{{\rm triv}},\\
0,\ \ \ \ &{\rm if} \ \chi^p\ne\chi_{{\rm triv}}.
\end{array}\right.}
\end{align*}
Hence the proof of Theorem \ref{thm 1.2} is complete in this case.

{\sc Case 2.} $p\nmid\binom{k+r}{r}$. At first, we define a set
$A\subseteq\mathcal{O}_K^3$ by
\begin{equation*}
A:=\{(x, y, z)\in\mathcal{O}_K^3|
{\rm ord}(x)\geq \omega,\ {\rm ord}(y)\geq 1,\ {\rm ord}(z)\geq 1\},
\end{equation*}
where $\omega:=\frac{k+r+l}{p}$. Then by (\ref{3.2}), we have
\begin{align}\label{4.2}
Z_f(s, \chi, A)=q^{-i-j-k}\int_{\tilde A}\chi(acf(\pi^ix, \pi^jy, \pi^kz))
|f(\pi^ix, \pi^jy, \pi^kz)|^s|dxdydz|,
\end{align}
where
$$\tilde A=\{(x, y, z)\in K^3| (\pi^i x, \pi^j y, \pi^k z)\in\mathcal{O}_K^3,
{\rm ord}(\pi^ix)\geq \omega,\ {\rm ord}(\pi^jy)\geq 1,\ {\rm ord}(\pi^kz)\geq 1\}.$$
Now letting $i=\omega,\ j=k=1$, then Lemma \ref{lem 2.11} (1) gives us that
$\tilde A=\mathcal{O}_K^3$. So by (\ref{4.2}), we derive that
\begin{align}\label{4.3}
Z_f(s, \chi, A)=&q^{-(\omega+2)}\int_{\mathcal{O}_K^3}
\chi(ac(\pi^{k+r+l}x^p+\pi^{r+l}y^rz^l\mathbb{H}_r^k(\pi y, t\pi z-\pi y)))\nonumber\\
\times&|\pi^{k+r+l}x^p+\pi^{r+l}y^rz^l\mathbb{H}_r^k(\pi y, t\pi z-\pi y)|^s|dxdydz|\nonumber\\
=&q^{-(\omega+2)-(k+r+l)s}Z_f(s, \chi)\nonumber\\
=&q^{-(\omega+2)}\int_{\mathcal{O}_K^3}
\chi(ac(\pi^{k+r+l}x^p+\pi^{k+r+l}y^rz^l\mathbb{H}_r^k(y, tz-y)))\nonumber\\
\times&|\pi^{k+r+l}x^p+\pi^{k+r+l}y^rz^l\mathbb{H}_r^k(y, tz-y)|^s|dxdydz|\nonumber\\
=&q^{-(\omega+2)-(k+r+l)s}Z_f(s, \chi)
\end{align}
On the other hand, let $A^c$ be the complement of $A$ in $\mathcal{O}_K^3$,
then since $\mathcal{O}_K^3=A\cup A^c$, one has
\begin{equation}\label{eqno 4.4}
Z_f(s, \chi, A)=Z_f(s, \chi)-Z_f(s, \chi, A^c).
\end{equation}
By (\ref{4.3}) and (\ref{eqno 4.4}), we have
$$Z_f(s, \chi)=\dfrac{1}{1-q^{-(\omega+2)-(k+r+l)s}}Z_f(s, \chi, A^c).$$

Moreover, one easily derives that $A^c$ can be
decomposed as the disjoint union of the following seven sets:
\begin{align*}
A_1&=\{(x, y, z)|{\rm ord}(x)\geq \omega,\ {\rm ord}(y)=0,\ {\rm ord}(z)\geq 1\},\\
A_2&=\{(x, y, z)|{\rm ord}(x)\geq \omega,\ {\rm ord}(y)\geq 1,\ {\rm ord}(z)=0\},\\
A_3&=\{(x, y, z)|{\rm ord}(x)\geq \omega,\ {\rm ord}(y)=0,\ {\rm ord}(z)=0\},\\
A_4&=\{(x, y, z)|0\le{\rm ord}(x)<\omega,\ {\rm ord}(y)\geq 1,\ {\rm ord}(z)\geq 1\},\\
A_5&=\{(x, y, z)|0\le{\rm ord}(x)<\omega,\ {\rm ord}(y)=0,\ {\rm ord}(z)\geq 1\},\\
A_6&=\{(x, y, z)|0\le{\rm ord}(x)<\omega,\ {\rm ord}(y)\geq 1,\ {\rm ord}(z)=0\},\\
A_7&=\{(x, y, z)|0\le{\rm ord}(x)<\omega,\ {\rm ord}(y)=0,\ {\rm ord}(z)=0\}.
\end{align*}
It then follows that
\begin{equation}\label{eqno 4.5}
Z_f(s, \chi)=\dfrac{1}{1-q^{-(\omega+2)-(k+r+l)s}}
\sum_{i=1}^{7}Z_f(s, \chi, A_i).
\end{equation}
In what follows, we calculate the seven integrals
on the right hand side of (\ref{eqno 4.5}), respectively.

For $Z_f(s, \chi, A_1)$, we set $i=\omega,\ j=0,\ k=1$ in (\ref{eqno 3.1}).
Then by (\ref{3.2}) and Lemma \ref{lem 2.11}, we have
\begin{align*}
&Z_f(s, \chi, A_1)\\
=&q^{-(\omega+1)}\\
\times&\int_{B_1}\chi(ac(\pi^{k+r+l}x^p+\pi^ly^rz^l
\mathbb{H}_r^k(y, t\pi z-y)|\pi^{k+r+l}x^p+\pi^ly^rz^l
\mathbb{H}_r^k(y, t\pi z-y)|^s|dxdydz|\\
=&q^{-(\omega+1)-ls}\\
\times&\int_{B_1}\chi(ac(\pi^{k+r}x^p+y^rz^l
\mathbb{H}_r^k(y, t\pi z-y)|\pi^{k+r}x^p+y^rz^l
\mathbb{H}_r^k(y, t\pi z-y)|^s|dxdydz|,
\end{align*}
where
$$B_1:=\{(x, y, z)\in K^3| (\pi^{\omega} x, \pi y, z)\in\mathcal{O}_K^3,
{\rm ord}(\pi^{\omega}x)\geq \omega,\ {\rm ord}(y)=0,\ {\rm ord}(\pi z)\ge 1 \}.$$
Then Lemma \ref{lem 2.11} tells us that
$B_1=\mathcal{O}_{K}\times\mathcal{O}_{K}^{\times}\times\mathcal{O}_{K}$.

Now we make the change of variables of the form:
$(x, y, z)\mapsto (uw^{(k+r+l)/p}, w, vw)$.
Since $|w|=1$ for $w\in \mathcal{O}_K^{\times}$, we derive that
\begin{align*}
&Z_f(s, A_1)\\
=&q^{-(\omega+1)-ls}
\int_{\mathcal{O}_K^2}\chi(ac(\pi^{k+r}u^pw^{k+r+l}+w^{r+l}v^l
\mathbb{H}_r^k(w, t\pi vw-w)\\
\times&|\pi^{k+r}u^pw^{k+r+l}+w^{r+l}v^l\mathbb{H}_r^k(w, t\pi vw-w)|^s
|-w^{\frac{k+r+l}{p}+1}||dudvdw|\\
:=&q^{-(\omega+1)-ls}Z_{f_1}(s, \chi)
\int_{\mathcal{O}_K^{\times}}\chi(ac(w^{k+r+l}))|dw|,
\end{align*}
where $f_1(u, v)=\pi^{k+r}u^p+v^l\mathbb{H}_r^k(1, t\pi v-1)$.

By Lemma \ref{lem 2.9}, we get that
\begin{align*}
\int_{\mathcal{O}_K^{\times}}\chi(ac (w^{k+r+l}))|dw|
={\left\{\begin{array}{rl}
1-q^{-1},\ \ \ &{\rm if}\ \chi^{k+r+l}=\chi_{\rm triv},\\
0,\ \ \ &{\rm if}\ \chi^{k+r+l}\neq \chi_{\rm triv}.
\end{array}\right.}
\end{align*}
Hence it follows that
\begin{align*}
Z_f(s, \chi, A_1)={\left\{\begin{array}{rl}
(1-q^{-1})q^{-(\omega+1)-ls}Z_{f_1}(s, \chi),
\ \ \ &{\rm if}\ \chi^{k+r+l}=\chi_{\rm triv},\\
0,\ \ \ &{\rm if}\ \chi^{k+r+l}\neq \chi_{\rm triv}.
\end{array}\right.}
\end{align*}
Notice that
\begin{align*}
f_1(u, v)=&\pi^{k+r}u^p+v^l\sum_{i=0}^k(-1)^i
\dbinom{k+r}{k-i}\dbinom{i+r-1}{i}(\pi t)^{k-i}v^{k-i}\\
=&\pi^{k+r}u^p+(-1)^k\dbinom{k+r-1}{k}v^l+\sum_{i=0}^{k-1}(-1)^i
\dbinom{k+r}{k-i}\dbinom{i+r-1}{i}(\pi t)^{k-i}v^{k+l-i}.
\end{align*}
But $p\mid(k+r+l)$ and $p\nmid l$, so $p\nmid (k+r)$.
On the other hand, we have
$$\dbinom{k+r-1}{k}=\dbinom{k+r-1}{r-1}=\dfrac{k+r}{r}\dbinom{k+r}{r}.$$
Then $p\nmid\binom{k+r-1}{k}$ because $p\nmid(k+r)\binom{k+r}{r}$, which
implies that $f_1$ is a polynomial of form (\ref{eqno 1.2}).
Thus by Theorem \ref{thm 1.1}, we get that
\begin{align}\label{4.6}
Z_f(s, \chi, A_1)=&(1-q^{-1})q^{-(\omega+1)-ls}Z_{f_1}(s, \chi)\nonumber\\
=&\dfrac{F_1(q^{-s})}{(1-q^{-1-s})(1-q^{-p-l-pls})},
\end{align}
where $F_1(x)$ is a polynomial with complex coefficients depending on $\chi$.

For $Z_f(s, \chi, A_2)$, setting $i=\omega,\ j=1$ and $k=0$ in (\ref{eqno 3.1})
and by (\ref{3.2}), one has
\begin{align*}
&Z_f(s, \chi, A_2)\\
=&q^{-(\omega+1)}\int_{B_2}\chi(ac(\pi^{k+r+l}x^p+\pi^ry^rz^l
\mathbb{H}_r^k(\pi y, tz-\pi y)))\\
\times&|\pi^{k+r+l}x^p+\pi^ry^rz^l
\mathbb{H}_r^k(\pi y, tz-\pi y)|^s|dxdydz|\\
=&q^{-(\omega+1)-rs}\\
\times&\int_{B_2}\chi(ac(\pi^{k+r}x^p+y^rz^l
\mathbb{H}_r^k(\pi y, tz-\pi y)|\pi^{k+r}x^p+y^rz^l
\mathbb{H}_r^k(\pi y, tz-\pi y)|^s|dxdydz|,
\end{align*}
where
$$B_2=\{(x, y, z)\in K^3| (\pi^{\omega} x, \pi y, z)\in\mathcal{O}_K^3,
{\rm ord}(\pi^{\omega}x)\geq \omega,\ {\rm ord}(\pi y)\ge 1,\ {\rm ord}(z)=0 \}.$$
Then Lemma \ref{lem 2.11} gives us that
$B_2=\mathcal{O}_K^{2}\times\mathcal{O}_K^{\times}$.

We make the change of variables of the form:
$(x, y, z)\mapsto (uw^{(k+r+l)/p}, vw, w)$. Then we deduce that
\begin{align*}
&Z_f(s, \chi, A_2)\\
=&q^{-\omega-1-rs}\int_{B_2}\chi(ac(\pi^{k+l}w^{k+r+l}u^p+w^{r+l}v^r
\mathbb{H}_r^k(\pi vw, tw-\pi vw)))\\
\times&|\pi^{k+l}w^{k+r+l}u^p+w^{r+l}v^r
\mathbb{H}_r^k(\pi vw, tw-\pi vw)|^s|w^{\frac{k+r+l}{p}+1}||dudvdw|,\\
:=&q^{-(\omega+1)-rs}Z_{f_2}(s, \chi)
\int_{\mathcal{O}_K^{\times}}\chi(ac(w^{k+r+l}))|dw|,
\end{align*}
where $f_2(u, v)=\pi^{k+l}u^p+v^r\mathbb{H}_r^k(\pi v, t-\pi v)$, and
the last equality is because $|w|=1$ for $w\in \mathcal{O}_K^{\times}$.
Using Lemma \ref{lem 2.9}, we have
\begin{align*}
Z_f(s, \chi, A_2)={\left\{\begin{array}{rl}
(1-q^{-1})q^{-(\omega+1)-rs}Z_{f_2}(s, \chi),
\ \ \ &{\rm if}\ \chi^{k+r+l}=\chi_{\rm triv},\\
0,\ \ \ &{\rm if}\ \chi^{k+r+l}\neq \chi_{\rm triv}.
\end{array}\right.}
\end{align*}
On the other hand,
\begin{align*}
f_2(u, v)=&\pi^{k+l}u^p+v^r
\sum_{i=0}^k(-1)^i\dbinom{k+r}{k-i}\dbinom{i+r-1}{i}
t^{k-i}\pi^iv^i\\
=&\pi^{k+l}u^p+t^k\dbinom{k+r}{k}v^r+
\sum_{i=1}^k(-1)^i\dbinom{k+r}{k-i}\dbinom{i+r-1}{i}
t^{k-i}\pi^iv^{r+i}.
\end{align*}
Since $p\nmid\binom{k+r}{k}$ and $t^k\in \mathcal{O}_K^{\times}$,
so $f_2$ is a polynomial of form (\ref{eqno 1.2}). Hence Theorem \ref{thm 1.1}
tells us that
\begin{align}\label{4.7}
Z_f(s, \chi, A_2)=&(1-q^{-1})q^{-\omega-1-rs}Z_{f_2}(s, \chi)\nonumber\\
=&\dfrac{F_2(q^{-s})}{(1-q^{-1-s})(1-q^{-p-r-prs})}
\end{align}
where $F_2(x)$ is a polynomial with complex coefficients depending on $\chi$.

For $Z_f(s, \chi, A_3)$, we set $i=\omega,\ j=k=0$ in (\ref{eqno 3.1}).
Then by (\ref{3.2}), one has
$$Z_f(s, \chi, A_3)=q^{-\omega}Z_{f_3}(s, \chi, B_3),$$
where $f_3(x, y, z)=\pi^{k+r+l}x^p+y^rz^l\mathbb{H}_r^k(y, tz-y)$,
and
$$B_3=\{(x, y, z)\in K^3| (\pi^{\omega} x, y, z)\in\mathcal{O}_K^3,
{\rm ord}(\pi^{\omega}x)\geq \omega,\ {\rm ord}(y)={\rm ord}(z)=0 \}.$$
By Lemma \ref{lem 2.10}, we derive that
$B_3=\mathcal{O}_K\times(\mathcal{O}_K^{\times})^2$, which implies
that $\bar{B_3}=\mathbb{F}_q\times(\mathbb{F}_q^{\times})^2$.
Using Lemma \ref{lem 2.1}, one has
\begin{equation*}
\dfrac{\partial \bar{f}_3}{\partial y}(x, y, z)
=\bar{r}\overline{\dbinom{k+r}{r}}y^{r-1}z^l(\bar{t}z-y)^k.
\end{equation*}
If $P=(x, y, z)\in {\rm Sing}_{\bar{f}_3}(\mathbb{F}_q)$, then
$\frac{\partial \bar{f}_3}{\partial y}=\bar{0}$
gives us that $zy(\bar{t}z-y)=\bar{0}$. We claim $y=\bar 0$ or $z=\bar 0$.
In fact, if the claim is not true, then we must have $y=\bar{t}z$. Thus
$\bar{f}_3(x, y, z)={\bar{t}z}^{k+r+l}\ne\bar{0}$, which is contracted to the
assumption $P\in {\rm Sing}_{\bar{f}_3}(\mathbb{F}_q)$. So the
claim is proved. But $\bar{B_3}=\mathbb{F}_q\times(\mathbb{F}_q^{\times})^2$.
Hence we deduce that
${\rm Sing}_{\bar{g}}(\mathbb{F}_q)\bigcap \bar{B_3}=\emptyset$,
which tells us that $S(g, B_3)=\emptyset$. By Lemma \ref{lem 2.2},
we obtain that
\begin{equation}\label{eqno 4.8}
Z_f(s, \chi, A_3)=\dfrac{F_3(q^{-s})}{1-q^{-1-s}},
\end{equation}
where $F_3(x)$ is a polynomial with complex coefficients depending on $\chi$.

In what follows, let $a$ be an integer with $0\le a<\omega$.
For any integer $b$ with $4\le b\le 7$, we define the set $A_b^a$ by
$$A_b^a:=\{(x, y, z)\in A_b|{\rm ord}(x)=a\}.$$
Then it follows that $A_b=\bigcup\limits_{0\le a<\omega}A_b^a$ and
\begin{equation}\label{eqno 4.9}
Z_f(s, \chi, A_b)=\sum_{a=0}^{\omega-1}{Z_f(s, \chi, A_b^a)}.
\end{equation}
For $Z_f(s, \chi, A_4^a)$, we set $i=a,\ j=k=1$ in (\ref{eqno 3.1}).
Therefore by (\ref{3.2}) and (\ref{eqno 4.9}), we get that
\begin{align*}
Z_f(s, \chi, A_4)=&\sum_{a=0}^{\omega-1}{Z_f(s, \chi, A_4^a)}\\
=&\sum_{a=0}^{\omega-1}q^{-a-2}\int_{B_4^a}
\chi(ac(\pi^{ap}x^p+\pi^{k+r+l}y^rz^l\mathbb{H}_r^k(y, tz-y)))\\
\times&|\pi^{ap}x^p+\pi^{k+r+l}y^rz^l\mathbb{H}_r^k(y, tz-y)|^s|dxdydz|\\
=&\sum_{a=0}^{\omega-1}q^{-a-2-aps}\int_{B_4^a}
\chi(ac(x^p+\pi^{k+r+l-ap}y^rz^l\mathbb{H}_r^k(y, tz-y)))\\
\times&|x^p+\pi^{k+r+l-ap}y^rz^l\mathbb{H}_r^k(y, tz-y)|^s|dxdydz|,
\end{align*}
where
$$B_4^a=\{(x, y, z)\in K^3| (\pi^{a} x, \pi y, \pi z)\in\mathcal{O}_K^3,
{\rm ord}(\pi^{a}x)=a,\ {\rm ord}(y)\ge 1,\ {\rm ord}(z)\ge 1 \}.$$
So Lemma \ref{lem 2.11} tells us that
$B_4^a=\mathcal{O}_K^{\times}\times\mathcal{O}_K^2$. Since
$x\in \mathcal{O}_K^{\times}$, we have
$$|x^p+\pi^{k+r+l-ap}y^rz^l\mathbb{H}_r^k(y, tz-y)|=1$$
for any $0\leq a<\omega$ and $(x, y, z)\in B_4^a$.
Hence we arrive at
\begin{align}\label{4.10}
Z_f(s, \chi, A_4)=&\sum_{a=0}^{\omega-1}q^{-a-2-aps}\int_{B_4^a}
\chi(ac(x^p+\pi^{k+r+l-ap}y^rz^l\mathbb{H}_r^k(y, tz-y)))|dxdydz|\nonumber\\
:=&F_4(q^{-s}),
\end{align}
where $F_4(x)$ is a polynomial with complex coefficients depending on $\chi$.

For $Z_f(s, \chi, A_5^a)$, we set $i=a,\ j=0,\ k=1$ in (\ref{eqno 3.1}).
Then from (\ref{3.2}) and (\ref{eqno 4.9}), one has
\begin{align}\label{4.11}
Z_f(s, \chi, A_5)=&\sum_{a=0}^{\omega-1}{Z_f(s, \chi, A_5^a)}\nonumber\\
=&\sum_{a=0}^{\omega-1}q^{-a-1}Z_{f_{5, a}}(s, \chi, B_5^a),
\end{align}
where
$$f_{5, a}(x, y, z):=\pi^{ap}x^p+\pi^ly^rz^l
\sum\limits_{i=0}^k(-1)^i\dbinom{k+r}{k-i}\dbinom{i+r-1}{i}
(\pi t)^{k-i}y^iz^{k-i}$$
and
$$B_5^a=\{(x, y, z)\in K^3| (\pi^{a} x, y, \pi z)\in\mathcal{O}_K^3,
{\rm ord}(\pi^{a}x)=a,\ {\rm ord}(y)=0, {\rm ord}(z)\ge 1 \}.$$
Then Lemma \ref{lem 2.11} gives us that
$B_5^a=(\mathcal{O}_K^{\times})^2\times\mathcal{O}_K$.
For $Z_{f_{5, a}}(s, \chi, B_5^a)$, we make the change of variables of
the form: $(x, y, z)\mapsto (uw^{(k+r+l)/p}, w, vw)$, then we have that
\begin{align*}
Z_{f_{5, a}}(s, \chi, B_5^a)={\left\{\begin{array}{rl}
(1-q^{-1})Z_{\tilde{f}_{5, a}}(s, \chi, C_5^a),
\ \ \ &{\rm if}\ \chi^{k+r+l}=\chi_{\rm triv},\\
0,\ \ \ &{\rm if}\ \chi^{k+r+l}\neq \chi_{\rm triv},
\end{array}\right.}
\end{align*}
where
$$\tilde{f}_{5, a}(u, v):=\pi^{ap}u^p+\pi^lv^l
\sum\limits_{i=0}^k(-1)^i\dbinom{k+r}{k-i}\dbinom{i+r-1}{i}
(\pi t)^{k-i}v^{k-i}$$
and $C_5^a=\mathcal{O}_K^{\times}\times\mathcal{O}_K$.
Thus by Lemma \ref{lem 2.12}, we derive that
\begin{align}\label{4.12}
Z_{\tilde{f}_{5, a}}(s, \chi, C_5^a)={\left\{\begin{array}{rl}
F_{5, a}(q^{-s}),\ \ &0\le a<\dfrac{l}{p},\\
\dfrac{\tilde F_{5, a}(q^{-s})}{1-q^{-1-s}},\ \ &{\dfrac{l}{p}\le a<\omega},
\end{array}\right.}
\end{align}
where $F_{5, a}(x), \tilde F_{5, a}(x)$ are polynomials with complex coefficients
depending on $\chi$.

Put (\ref{4.12}) into (\ref{4.11}), one derives that
\begin{equation}\label{eqno 4.13}
Z_f(s, \chi, A_5)=\dfrac{F_5(q^{-s})}{1-q^{-1-s}},
\end{equation}
where $F_5(x)$ is a polynomial with complex coefficients depending on $\chi$.

For $Z_f(s, \chi, A_6^a)$, we set $i=a,\ j=1,\ k=0$ in (\ref{eqno 3.1}).
By (\ref{3.2}) and (\ref{eqno 4.9}), one deduces that
\begin{align}\label{4.14}
Z_f(s, \chi, A_6)
&=\sum_{a=0}^{\omega-1}{Z_f(s, \chi, A_6^a)}\nonumber\\
&=\sum_{a=0}^{\omega-1}q^{-1-a}Z_{f_{6, a}}(s, \chi, B_6^a),
\end{align}
where
$$f_{6, a}(x, y, z):=\pi^{ap}u^p+\pi^ry^rv^l
\sum\limits_{i=0}^k(-1)^i\dbinom{k+r}{k-i}\dbinom{i+r-1}{i}
t^{k-i}\pi^iy^iv^{k-i}$$
and
$$B_6^a=\{(x, y, z)\in K^3| (\pi^{a} x, \pi y, z)\in\mathcal{O}_K^3,
{\rm ord}(\pi^{a}x)=a,\ {\rm ord}(y)\ge 1, {\rm ord}(z)=0 \}.$$
By Lemma \ref{lem 2.11}, we have
$B_6^a=\mathcal{O}_K^{\times}\times\mathcal{O}_K\times\mathcal{O}_K^{\times}$.
For $Z_{f_{6, a}}(s, B_6^a)$, we make the change of variables of
the form: $(x, y, z)\mapsto (uw^{(k+r+l)/p}, vw, w)$, then it follows that
\begin{align*}
Z_{f_{6, a}}(s, \chi, B_6^a)={\left\{\begin{array}{rl}
(1-q^{-1})Z_{\tilde{f}_{6, a}}(s, \chi, C_6^a),
\ \ \ &{\rm if}\ \chi^{k+r+l}=\chi_{\rm triv},\\
0,\ \ \ &{\rm if}\ \chi^{k+r+l}\ne \chi_{\rm triv},
\end{array}\right.}
\end{align*}
where
$$\tilde{f}_{6, a}(u, v):=\pi^{ap}u^p+\pi^rv^r
\sum\limits_{i=0}^k(-1)^i\dbinom{k+r}{k-i}\dbinom{i+r-1}{i}
t^{k-i}\pi^iv^i$$
and $C_6^a=\mathcal{O}_K^{\times}\times\mathcal{O}_K$.
Then using Lemma \ref{lem 2.12}, we get that
\begin{align}\label{4.15}
Z_{\tilde{f}_{6, a}}(s, \chi, C_6^a)={\left\{\begin{array}{rl}
F_{6, a}(q^{-s}),\ \ &0\le a<\dfrac{r}{p},\\
\dfrac{\tilde F_{6, a}(q^{-s})}{1-q^{-1-s}},\ \ &{\dfrac{r}{p}\le a<\omega},
\end{array}\right.}
\end{align}
where $F_{6, a}(x), \tilde F_{6, a}(x)$ are polynomials with complex coefficients
depending on $\chi$.
Put (\ref{4.15}) into (\ref{4.14}), one derives that
\begin{equation}\label{eqno 4.16}
Z_f(s, \chi, A_6)=\dfrac{F_6(q^{-s})}{1-q^{-1-s}},
\end{equation}
where $F_6(x)$ is a polynomial with complex coefficients depending on $\chi$.

For $Z_f(s, \chi, A_7^a)$, we set $i=a,\ j=k=0$ in (\ref{eqno 3.1}).
Then it follows from (\ref{3.2}) and (\ref{eqno 4.9}) that
\begin{align}\label{4.17}
Z_f(s, \chi, A_7)=&\sum_{a=0}^{\omega-1}{Z_f(s, \chi, A_7^a)}\nonumber\\
=&\sum_{a=0}^{\omega-1}q^{-a}{Z_{f_{7, a}}(s, \chi, B_7^a)},
\end{align}
where
$$f_{7, a}(x, y, z):=\pi^{ap}x^p+
y^rz^l\sum_{i=0}^{k}{\dbinom{k+r}{i+r}y^i(tz-y)^{k-i}}$$
and
$$B_7^a=\{(x, y, z)\in K^3| (\pi^{a} x, y, z)\in\mathcal{O}_K^3,
{\rm ord}(\pi^{a}x)=a,\ {\rm ord}(y)={\rm ord}(z)=0 \}.$$
Notice that Lemma \ref{lem 2.11} tells us that 
$B_7^a=(\mathcal{O}_K^{\times})^3$,
which infers that $\bar{B_7^a}=(\mathbb{F}_q^{\times})^3$.

Let $a$ be an integer with $0<a<\omega$. If
$P=(x, y, z)\in {\rm Sing}_{\bar{f}_{7, a}}(\mathbb{F}_q)$, 
then
$$
\frac{\partial \bar{f}_{7, a}}{\partial y}
=\bar{r}\overline{\dbinom{k+r}{r}}y^{r-1}z^l(tz-y)^k=\bar{0}
$$
that yields that $yz=\bar{0}$ (see the discussion for $Z_f(s, \chi, A_3)$).
Thus ${\rm Sing}_{\bar{f}_{7, a}}(\mathbb{F}_q)\bigcap \bar{B_7^a}=\emptyset$,
which implies that $S(f_{7, a}, B_7^a)=\emptyset$. Then by Lemma \ref{lem 2.2},
one has
\begin{equation}\label{eqno 4.18}
Z_{f_{7, a}}(s, \chi, B_7^a)=\dfrac{F_{7, a}(q^{-s})}{1-q^{-1-s}},
\end{equation}
where $F_{7, a}(x)$ is a polynomial with complex coefficients depending on $\chi$.

When $a=0$, by Lemma \ref{lem 2.13}, we arrive at
\begin{equation}\label{eqno 4.19}
Z_{f_{7, 0}}(s, \chi, B_7^a)=\dfrac{F_{7, 0}(q^{-s})}
{(1-q^{-1-s})(1-q^{-p-(k+1)-p(k+1)s})},
\end{equation}
where $F_{7, 0}(x)$ is a polynomial with complex coefficients depending on $\chi$.
Then by (\ref{4.17}), (\ref{eqno 4.18}) and (\ref{eqno 4.19}),
we obtain that
\begin{equation}\label{eqno 4.20}
Z_f(s, \chi, A_7)=\dfrac{F_7(q^{-s})}{(1-q^{-1-s})(1-q^{-p-(k+1)-p(k+1)s})},
\end{equation}
where $F_7(x)$ is a polynomial with complex coefficients depending on $\chi$.

Finally, combining (\ref{eqno 4.5}), (\ref{4.6}), (\ref{4.7}),
(\ref{eqno 4.8}), (\ref{4.10}), (\ref{eqno 4.13}), (\ref{eqno 4.16})
with (\ref{eqno 4.20}) gives us the desired result. So Theorem \ref{thm 1.2}
is proved in this case.

This finishes the proof of Theorem \ref{thm 1.2}. \hfill$\Box$\\

\section{\bf Examples and remarks}

In the final section, we first present some examples to illustrate the
validity of Theorems \ref{thm 1.1} and \ref{thm 1.2}. Consequently,
we give a remark as a conclusion of this paper.

Let $K$ be a non-archimedean local field with characteristic $p$ and
$\mathbb{F}_p$ being its residue field. Put $\chi=\chi_{\rm triv}$ and $t:=p^{-s}$.
Let $f(x, y, z)=x^p+y^rz^l\mathbb{H}_r^k(y, z-y)$ be a hybrid polynomial.
Now we give the following examples.

\begin{exm}\label{exm 5.1}
Let $p=3$, $k=6$, $r=4$ and $l=2$. Then it follows that $3\mid\binom{10}{4}$,
and $f(x, y, z)=x^3+y^9z^3$. Hence we have
$$Z_f(s, \chi)=\dfrac{1-3^{-1}}{1-3^{-1}t^3}.$$
\end{exm}

\begin{exm}\label{exm 5.2}
Let $p=5$, $k=5$, $r=2$ and $l=3$, then $\omega=2$.
Moreover, we have $5\nmid\binom{7}{2}$, and
$$f(x, y, z)=x^5+y^2z^8+y^5z^5-y^7z^3.$$
By (\ref{eqno 4.5}), one has
$$Z_f(s, \chi)=\dfrac{1}{1-5^{-4}t^{10}}\sum_{i=1}^{7}Z_f(s, \chi, A_i),$$
where $A_i$ are sets defined above (\ref{eqno 4.5}). Then we can calculate the seven
integrals explicitly as follows:

For $Z_f(s, \chi, A_1)$, one has
\begin{equation*}
Z_f(s, \chi, A_1)=5^{-3}(1-5^{-1})t^3Z_{g_1}(s, \chi),
\end{equation*}
where $g_1(x, z)=\pi^7x^5-z^3+\pi^2z^5+\pi^5z^8$.
It is easy to see that $g_1$ is a polynomial of form (\ref{eqno 1.2}).
By (\ref{eqno 3.3}), we derive that
\begin{equation*}
Z_{g_1}(s, \chi)=Z_{g_1}\big(s, \chi, (\mathcal{O}_K^{\times})^2\big)
+\sum_{i=1}^5Z_{g_1}\big(s, \chi, S(\Delta_{1, i}\bigcap (\mathbb{N}^2\setminus \{0\}))\big),
\end{equation*}
where $\Delta_{1, 1}=\{(0, a)|a\in \mathbb{R}^{+}\}$,
$\Delta_{1, 2}=\{(3b, a+5b)|a, b\in \mathbb{R}^{+}\}$,
$\Delta_{1, 3}=\{(3a, 5a)|a\in \mathbb{R}^{+}\}$,
$\Delta_{1, 4}=\{(3a+b, 5a)|a, b\in \mathbb{R}^{+}\}$,
$\Delta_{1, 5}=\{(a, 0)|a\in \mathbb{R}^{+}\}$.
Then we deduce that
\begin{align*}
Z_{g_1}\big(s, \chi, (\mathcal{O}_K^{\times})^2\big)=&(1-5^{-1})^2,\\
Z_{g_1}\big(s, \chi, S(\Delta_{1, 1}\bigcap (\mathbb{N}^2\setminus \{0\}))\big)
=&(1-5^{-1})t^3\Big(5^{-3}t^4+(5^{-2}-5^{-3})t^3+5^{-1}-5^{-2}\Big),\\
Z_{g_1}\big(s, \chi, S(\Delta_{1, 2}\bigcap (\mathbb{N}^2\setminus \{0\}))\big)
=&\dfrac{(1-5^{-1})5^{-8}t^{15}}{(1-5^{-8}t^{18})(1-5^{-1}t)}
\Big((5^{-5}-5^{-6})t^9-(5^{-5}-5^{-6})t^7\\
+&(5^{-4}-5^{-5})t^6-5^{-4}t^5+5^{-3}t^4+(5^{-2}-5^{-3})t^3\\
-&(5^{-2}-5^{-3})t+5^{-1}-5^{-2}\Big),\\
Z_{g_1}\big(s, \chi, S(\Delta_{1, 3}\bigcap (\mathbb{N}^2\setminus \{0\}))\big)
=&\dfrac{5^{-8}t^{15}(1-5^{-1})^2}{1-5^{-8}t^{15}},\\
Z_{g_1}\big(s, \chi, S(\Delta_{1, 4}\bigcap (\mathbb{N}^2\setminus \{0\}))\big)
=&\dfrac{(1-5^{-1})5^{-9}t^{15}}{1-5^{-8}t^{15}}
\Big(5^{-7}t^{12}+5^{-5}t^9+5^{-4}t^6+5^{-2}t^3+1\Big)\\
Z_{g_1}\big(s, \chi, S(\Delta_{1, 5}\bigcap (\mathbb{N}^2\setminus \{0\}))\big)
=&5^{-1}-5^{-2}.
\end{align*}
Thus all the candidate poles in Theorem \ref{thm 1.1} are indeed the poles
of $Z_{g_1}(s, \chi)$.

For $Z_f(s, \chi, A_2)$, one has
\begin{equation*}
Z_f(s, \chi, A_2)=5^{-3}(1-5^{-1})t^2Z_{g_2}(s, \chi),
\end{equation*}
where $g_2(x, y)=\pi^8x^5+y^2+\pi^3y^5-\pi^5y^7$,
which is also a a polynomial of form (\ref{eqno 1.2}).
By (\ref{eqno 3.3}), we get that
\begin{equation*}
Z_{g_2}(s, \chi)=Z_{g_2}\big(s, \chi, (\mathcal{O}_K^{\times})^2\big)
+\sum_{i=1}^5Z_{g_2}\big(s, \chi, S(\Delta_{2, i}\bigcap (\mathbb{N}^2\setminus \{0\}))\big),
\end{equation*}
where $\Delta_{2, 1}=\{(0, a)|a\in \mathbb{R}^{+}\}$,
$\Delta_{2, 2}=\{(2b, a+5b)|a, b\in \mathbb{R}^{+}\}$,
$\Delta_{2, 3}=\{(2a, 5a)|a\in \mathbb{R}^{+}\}$,
$\Delta_{2, 4}=\{(2a+b, 5a)|a, b\in \mathbb{R}^{+}\}$,
$\Delta_{2, 5}=\{(a, 0)|a\in \mathbb{R}^{+}\}$.
Then it follows that
\begin{align*}
Z_{g_2}\big(s, \chi, (\mathcal{O}_K^{\times})^2\big)=&(1-5^{-1})^2,\\
Z_{g_2}\big(s, \chi, S(\Delta_{2, 1}\bigcap (\mathbb{N}^2\setminus \{0\}))\big)
=&\dfrac{(1-5^{-1})t^2}{1-5^{-1}t}\Big(5^{-6}t^7+(5^{-4}-2\times5^{-5})t^6-(5^{-4}-5^{-5})t^5\\
+&(5^{-3}-5^{-4})t^4-(5^{-3}-5^{-4})t^3+(5^{-2}-5^{-3})t^2\\
-&(5^{-2}-5^{-3})t+5^{-1}-5^{-2}\Big),\\
Z_{g_2}\big(s, \chi, S(\Delta_{2, 2}\bigcap (\mathbb{N}^2\setminus \{0\}))\big)
=&\dfrac{(1-5^{-1})t^2}{1-5^{-7}t^{10}}\Big(-5^{-16}t^{22}+5^{-16}t^{21}+(5^{-14}-5^{-15})t^{20}\\
-&(5^{-14}-5^{-15})t^{19}+(5^{-13}-5^{-14})t^{18}-(5^{-13}-5^{-14})t^{17}\\
+&(5^{-12}-5^{-13})t^{16}+5^{-6}t^7+(5^{-4}-2\times5^{-5})t^6\\
-&(5^{-4}-5^{-5})t^5+(5^{-3}-5^{-4})t^4-(5^{-3}-5^{-4})t^3\\
+&(5^{-2}-5^{-3})t^2-(5^{-2}-5^{-3})t+5^{-1}-5^{-2}\Big),\\
Z_{g_2}\big(s, \chi, S(\Delta_{2, 3}\bigcap (\mathbb{N}^2\setminus \{0\}))\big)
=&\dfrac{5^{-7}t^{10}(1-5^{-1})^2}{1-5^{-7}t^{10}},\\
Z_{g_2}\big(s, \chi, S(\Delta_{2, 4}\bigcap (\mathbb{N}^2\setminus \{0\}))\big)
=&\dfrac{(1-5^{-1})5^{-8}t^{10}}{1-5^{-7}t^{10}}
\Big(5^{-6}t^8+5^{-5}t^6+5^{-3}t^4+5^{-2}t^2+1\Big),\\
Z_{g_2}\big(s, \chi, S(\Delta_{2, 5}\bigcap (\mathbb{N}^2\setminus \{0\}))\big)
=&5^{-1}-5^{-2}.
\end{align*}
Hence we still have that all the candidate poles in Theorem \ref{thm 1.1}
are indeed the poles of $Z_{g_2}(s, \chi)$.

Moreover, we deduce that
\begin{align*}
Z_f(s, \chi, A_3)=&\dfrac{(1-5^{-1})}{1-5^{-1}t}\Big(5^{-4}t+5^{-2}-2\times5^{-3}\Big),\\
Z_f(s, \chi, A_4)=&(1-5^{-1})\Big(5^{-3}t^5+5^{-2}\Big),\\
Z_f(s, \chi, A_5)=&(1-5^{-1})^2\Big(5^{-3}t^5+(5^{-2}-5^{-3})t^3+5^{-1}\Big),\\
Z_f(s, \chi, A_6)=&(1-5^{-1})^2\Big(5^{-4}t^5+(5^{-3}-5^{-4})t^4+(5^{-2}-5^{-3})t^2+5^{-1}\Big).
\end{align*}

For $Z_f(s, \chi, A_7)$, we have
\begin{align*}
Z_f(s, \chi, A_7)=&\dfrac{(1-5^{-1})^2}{1-5^{-1}t}
\Big(-(5^{-1}+2\times5^{-2}-5^{-3}+5^{-4})t+1-5^{-1}-2\times5^{-2}\Big)\\
+&(5^{-2}-5^{-3})t^5Z_{g_7}(s, \chi),
\end{align*}
where $g_7(x, y)=x^5-2\pi y^6-\pi^2y^7$,
which is a polynomial of form (\ref{eqno 1.2}).
Using (\ref{eqno 3.3}), one derives that
\begin{equation*}
Z_{g_7}(s, \chi)=Z_{g_7}\big(s, \chi, (\mathcal{O}_K^{\times})^2\big)
+\sum_{i=1}^5Z_{g_7}\big(s, \chi, S(\Delta_{7, i}\bigcap (\mathbb{N}^2\setminus \{0\}))\big),
\end{equation*}
where $\Delta_{7, 1}=\{(0, a)|a\in \mathbb{R}^{+}\}$,
$\Delta_{7, 2}=\{(6b, a+5b)|a, b\in \mathbb{R}^{+}\}$,
$\Delta_{7, 3}=\{(6a, 5a)|a\in \mathbb{R}^{+}\}$,
$\Delta_{7, 4}=\{(6a+b, 5a)|a, b\in \mathbb{R}^{+}\}$,
$\Delta_{7, 5}=\{(a, 0)|a\in \mathbb{R}^{+}\}$.
Then we obtain that
\begin{align*}
Z_{g_7}\big(s, \chi, (\mathcal{O}_K^{\times})^2\big)=&(1-q^{-1})^2,\\
Z_{g_7}\big(s, \chi, S(\Delta_{7, 1}\bigcap (\mathbb{N}^2\setminus \{0\}))\big)
=&5^{-1}-5^{-2},\\
Z_{g_7}\big(s, \chi, S(\Delta_{7, 2}\bigcap (\mathbb{N}^2\setminus \{0\}))\big)
=&\dfrac{(1-5^{-1})5^{-12}t^{30}}{1-5^{11}t^{30}}\\
\times&(5^{-10}t^25+5^{-8}t^{20}+5^{-6}t^{15}+5^{-4}t^{10}+5^{-2}t^5+1),\\
Z_{g_7}\big(s, \chi, S(\Delta_{7, 3}\bigcap (\mathbb{N}^2\setminus \{0\}))\big)
=&\dfrac{5^{-11}t^{30}(1-5^{-1})^2}{1-5^{-11}t^{30}},\\
Z_{g_7}\big(s, \chi, S(\Delta_{7, 4}\bigcap (\mathbb{N}^2\setminus \{0\}))\big)
=&\dfrac{(1-5^{-1})5^{-12}t^{31}}{1-5^{-7}t^{10}}\\
\times&(5^{-9}t^8+5^{-5}t^{24}+5^{-7}t^{18}+5^{-5}t^{12}+5^{-3}t^{6}+1),\\
Z_{g_7}\big(s, \chi, S(\Delta_{7, 5}\bigcap (\mathbb{N}^2\setminus \{0\}))\big)
=&(5^{-1}-5^{-2})t.
\end{align*}
Hence all the candidate poles in Theorem \ref{thm 1.1}
are indeed the poles of $Z_{g_7}(s, \chi)$. It follows that 
all the possible poles provided in Theorem \ref{thm 1.2} (2) 
are indeed poles of $Z_f(s, \chi)$.
\end{exm}

{\bf Remark 5.3.}
In the proof of Theorem \ref{thm 1.2}, we used Theorem \ref{thm 1.1}
to calculate $Z_{f_2}(s, \chi)$, where
$$f_2(u, v)=\pi^{k+l}u^p+
\sum_{i=0}^k(-1)^i\dbinom{k+r}{k-i}\dbinom{i+r-1}{i}t^{k-i}\pi^iv^{r+i}.$$
But if $r=1$, then 
\begin{equation*}
\bar{f}_2(u, v)=\bar{t}^k\overline{\dbinom{k+r}{k}}v.
\end{equation*}
Obviously, ${\rm Sing}_{\bar{f}_2}(\mathbb{F}_q)=\emptyset$,
which implies that $S(f_2, \mathcal{O}_K^2)=\emptyset$.
By Lemma \ref{lem 2.2}, we have
$$Z_f(s, \chi, A_2)=\dfrac{L(q^{-s})}{1-q^{-1-s}},$$
where $L(x)$ is a polynomial with complex coefficients 
depending on $\chi$. So when $r=1$, $-\frac{1}{p}-\frac{1}{r}$ 
is never a pole of $f$. This explains why there are only four
candidate poles in the case $r=1$ (see \cite{[YH]}). Likewise,
if $r=l=1$, then the same discussion yields that neither 
$-\frac{1}{p}-\frac{1}{r}$ nor $-\frac{1}{p}-\frac{1}{l}$ is 
a pole of $f$. So there are only three possible poles in that 
case. This coincides with the result in \cite{[LIS]}.


\begin{thebibliography}{99}
\bibitem{[De1]} J. Denef, The rationality of the Poincar$\acute{e}$ series
associated to the $p$-adic points on a variety, Invent Math.
{\bf 77} (1984), 1-23.

\bibitem{[De2]} J. Denef, Poles of $p$-adic complex powers and Newton
polyhedra, Nieuw Archief voor Wiskunde {\bf 13} (1995), 289-295.

\bibitem{[Ha1]} H. Hauser, Why the characteristic zero proof of resolution
of singularities fails in positive characteristic, manuscript, 2003.
Available at http://homepage.univie.ac.at/herwig.hauser/index.html.

\bibitem{[Ha2]} H. Hauser, On the problem of resolution of singularities
in positive characteristic (or: a proof we are still waiting for),
Bull. Amer. Math. Soc. {\bf 47} (2010), 1-30.

\bibitem{[Igu1]} J.I. Igusa, Complex powers and asymptotic expansions, I,
J. Reine Angew. Math. {\bf 268/269} (1974), 110-130.

\bibitem{[Igu2]} J.I. Igusa, Complex powers and asymptotic expansions, II,
J. Reine Angew. Math. {\bf 278/279} (1975), 307-321.

\bibitem{[Igu3]} J.I. Igusa, A stationary phase formula for
$p$-adic integrals and it's applications, Algebraic Geometry
and its applications, Springer-Verlag, 1994, pp. 175-194.

\bibitem{[LIS]} E. Le$\acute{o}$n-Cardenal, D. Ibadula and D. Segers,
Poles of the Igusa local zeta function of some hybrid polynomials,
Finite Fields Appl. {\bf 25} (2014), 37-48.

\bibitem{[M]} D. Meuser, A survey of Igusa's local zeta function,
Amer. J. Math. {\bf 138} (2016), 149-179.

\bibitem{[S]} D. Segers, Lower bound for the poles of Igusa's $p$-adic
zeta functions, Math. Ann. {\bf 336} (2006), 659-669.

\bibitem{[W]} A. Weil, Sur certains groupes d'op$\acute{e}$rateurs,
Acta Math. {\bf 111} (1964), 143-211.

\bibitem{[YH]} Q.Y. Yin and S.F. Hong, Igusa local zeta functions
of a class of hybrid polynomials, arXiv:1611.02111.

\bibitem{[ZG1]} W.A. Z$\acute{u}$$\tilde{n}$iga-Galindo, Igusa's local
zeta function of semiquasihomogeneous polynomials, Trans. Amer. Math. Soc.
{\bf 353} (2001), 3193-3207.

\bibitem{[ZG2]} W.A. Z$\acute{u}$$\tilde{n}$iga-Galindo, Local zeta functions
and Newton polyhedra, Nagoya Math J. {\bf 172} (2003), 31-58.


\end{thebibliography}
\end{document}